\documentclass[11pt,twoside,reqno]{amsart}
\usepackage{fullpage}
\usepackage{amssymb,amsfonts,amsmath,amsthm,bm}

\usepackage{ stmaryrd}
\theoremstyle{definition}
	\newtheorem{thm}{Theorem}
	\newtheorem{prop}[thm]{Proposition}

	\newtheorem{conj}[thm]{Conjecture}
	
\usepackage{enumerate}
\usepackage{enumitem}
\usepackage{hyperref}
\usepackage{textcomp}
\usepackage{color}

\usepackage[mathscr]{eucal}
\usepackage{bbm}

\usepackage{hyperref}
\usepackage[dvipsnames]{xcolor}
\newcommand\myshade{85}
\colorlet{mylinkcolor}{red}
\colorlet{mycitecolor}{blue}
\colorlet{myurlcolor}{Aquamarine}
\hypersetup{
	linkcolor  = mylinkcolor,
	citecolor  = mycitecolor,
	urlcolor   = myurlcolor!\myshade!black,
	colorlinks = true,
}

\usepackage{etoolbox}
\patchcmd{\section}{\scshape}{\bfseries}{}{}
\makeatletter
\renewcommand{\@secnumfont}{\bfseries}
\makeatother

\newcommand{\ZZ}{\mathbb{Z}}

\newcommand{\sC}{\mathscr{C}}
\newcommand{\sN}{\mathscr{N}}

\newcommand{\sS}{\mathscr{S}}

\DeclareMathOperator{\wt}{wt}

\numberwithin{equation}{section}
\allowdisplaybreaks

\newcommand{\ignore}[1]{}

\newcommand{\ii}{{\bf{i}}}

\usepackage{array}

\begin{document}
\thispagestyle{empty}
\title{Principal subspaces of basic modules for twisted affine Lie algebras, $q$-series multisums, and Nandi's identities}

\author{Katherine Baker}
\address{Department of Mathematics and Computer Science, Ursinus College,
601 E Main St, Collegeville, PA 19426}
\email{kabaker2@ursinus.edu}

\author{Shashank Kanade}
\address{Department of Mathematics, University of Denver, Denver, CO 80208}
\email{shashank.kanade@du.edu}

\author{Matthew C. Russell}
\address{Department of Mathematics, 
University of Illinois Urbana-Champaign,
Urbana, IL 61801}
\email{mcr39@illinois.edu}

\author{Christopher Sadowski}
\address{Department of Mathematics and Computer Science, Ursinus College,
601 E Main St, Collegeville, PA 19426}
\email{csadowski@ursinus.edu}

\begin{abstract} 
 We provide an observation relating several known and conjectured $q$-series
 identities to the theory of principal subspaces of basic modules for
 twisted affine Lie algebras. We also state and prove two new families of
 $q$-series identities. The first family provides quadruple sum
 representations for Nandi's identities, including a manifestly positive
 representation for the first identity. The second is a family of new mod 10
 identities connected with principal characters of level 4 integrable,
 highest-weight modules of $\mathrm{D}_4^{(3)}$. 
\end{abstract}

\maketitle

\renewcommand\contentsname{\textbf{Contents}}
\setcounter{tocdepth}{1}
\tableofcontents

\section{Introduction}

Principal subspaces of standard (i.e., highest-weight and integrable) modules for untwisted affine Lie algebras were introduced and studied by Feigin and Stoyanovsky~\cite{FS1,FS2}, and their study from a vertex-algebraic point of view has been developed by Calinescu, Capparelli, Lepowsky, and Milas~\cite{CLM1,CLM2,CalLM1,CalLM2,CalLM3}, and many others. In particular, the graded dimensions of principal subspaces are interesting due to their connection to various partition identities and recursions they satisfy. The study of principal subspaces for standard modules of twisted affine Lie algebras was initiated by Calinescu, Lepowsky, and Milas~\cite{CalLM4}, and further developed in works by Calinescu, Milas, Penn, and the fourth author~\cite{CalMPe,PS1,PS2,CPS}. The multigraded dimensions for principal subspaces of basic (i.e., standard module with highest weight $\Lambda_0$, see, for example, Carter's book \cite[P.\ 508]{Car-book}) modules for twisted affine Lie algebras are well-known, and have been studied in several of those papers~\cite{CalMPe,PS1,PS2}. In particular, they take the form
\begin{equation}\label{twistedsum}
\sum_{{\bf m} \in (\mathbb{Z}_{\ge 0})^d}\frac{q^{\frac{{\bf m}^T A[\nu] {\bf m}}{2}}}{(q^{\frac{k}{l_1}};q^{\frac{k}{l_1}})_{m_1}\cdots (q^{\frac{k}{l_d}};q^{\frac{k}{l_d}})_{m_d}}x_1^{m_1}\cdots x_d^{m_d}
\end{equation}
where $A[\nu]$ is a matrix obtained by ``folding" a Cartan matrix $A$ of type $\mathrm{A}$, $\mathrm{D}$, 
or $\mathrm{E}$ by a Dynkin Diagram automorphism of order $k$, and $l_1,\dots,l_d$ are the sizes of the orbits of various simple roots (this folding is defined generally by Penn, Webb, and the fourth author~\cite{PSW}). The matrices $A[\nu]$ are symmetrized Cartan matrices of types $\mathrm{B}$, $\mathrm{C}$, $\mathrm{F}$, $\mathrm{G}$, and, in the case $\mathrm{A}_{2n}^{(2)}$, of the tadpole Dynkin diagram.

Another recently active field of research is finding and proving Rogers--Ramanujan-type (multi)sum-to-product identities corresponding mainly to the principal characters of various affine Lie algebras.
Here, the principal characters refer to principally specialized characters divided by a certain factor depending on the affine Lie algebra in question. 
For more on this terminology, see the works of the second and third authors~\cite{KanRus-a22}, \cite{KanRus-cylindric} or Sills' textbook~\cite{Sil-book}.
This use of ``principal'' is not to be confused with the ``principal'' subspaces mentioned above; these correspond to completely different notions.
The second and third authors conjectured identities~\cite{KanRus-stair} regarding principal characters of level $2$ standard modules of $\mathrm{A}_9^{(2)}$, which were later proved by Bringmann, Jennings-Shaffer, and Mahlburg~\cite{BriJenMah} and Rosengren~\cite{Ros}; Takigiku and Tsuchioka~\cite{TakTsu-a22} provided various results on levels $5$ and $7$ of $\mathrm{A}_2^{(2)}$ and some conjectures on level $2$ standard modules of $\mathrm{A}_{13}^{(2)}$;
the authors~\cite{KanRus-a22} proved identities for all standard modules of $\mathrm{A}_2^{(2)}$. Andrews, Schilling, and Warnaar~\cite{ASW}, Corteel, Dousse, and Uncu~\cite{CorDouUnc}, Warnaar~\cite{War2}, the second and third authors~\cite{KanRus-cylindric}, and Tsuchioka~\cite{Tsu-a21level3} all provided conjectures and/or proved results on identities related to the standard modules of $\mathrm{A}_2^{(1)}$.
Finally, Griffin, Ono, and Warnaar~\cite{GriOnoWar} demonstrated  many identities for (not necessarily principal) characters for a variety of affine Lie algebra modules. For an excellent overview of Rogers--Ramanujan-type identities, we refer the reader to the textbook of Sills~\cite{Sil-book}.

This work grew out of the following observations: The graded dimension of the principal subspaces of the basic $\mathrm{A}_{2n}^{(2)}$ module was studied by Calinescu, Milas, and Penn~\cite{CalMPe} and is given by 
\begin{equation}\label{A2n-sum}
\sum_{{\bf m} \in (\mathbb{Z}_{\ge 0})^d}\frac{q^{\frac{{\bf m}^T T_d {\bf m}}{2}}}{(q;q)_{m_1}\cdots (q;q)_{m_d}}
\end{equation}
where $T_d$ is the Cartan matrix of the tadpole Dynkin diagram. Meanwhile, Calinescu, Penn, and the fourth author~\cite{CPS} conjectured that the graded dimension of the principal subspace of the level $k$ 
standard $\mathrm{A}_2^{(2)}$ module having highest weight $k\Lambda_0$ is given by the sum side of
\begin{equation}\label{stem-sum}
\sum_{{\bf m} \in (\mathbb{Z}_{\ge 0})^d}\frac{q^{\frac{{\bf m}^T 2T_d^{-1} {\bf m}}{2}}}{(q^2;q^2)_{m_1}\cdots (q^2;q^2)_{m_d}},
\end{equation}
which is the sum side in Stembridge's generalization of the G\"ollnitz--Gordon--Andrews identities~\cite{S} (see also ~\cite{BIS} and ~\cite{War1} for more general sums of this form). This conjectured graded dimension has been proved by Takenaka~\cite{Ta}. Similarly, Penn and the fourth author~\cite{PS1} showed that the graded dimension of the basic $\mathrm{D}_4^{(3)}$-module is given by
\begin{equation}
\label{eqn:d43}
\sum_{{\bf m} \in (\mathbb{Z}_{\ge 0})^2}\frac{q^{\frac{{\bf m}^T A[\nu]  {\bf m}}{2}}}{(q^3;q^3)_{m_1}(q;q)_{m_2}},
\end{equation}
where 
\begin{equation}
A[\nu] = \begin{bmatrix} 6 & -3\\ -3 & 2 \end{bmatrix}.
\end{equation}
Meanwhile, Penn, Webb, and the fourth author~\cite{PSW} constructed a principal subspace of a twisted module for a lattice vertex operator algebra whose graded dimension is
\begin{equation}
\label{eqn:d43inv}
\sum_{{\bf m} \in (\mathbb{Z}_{\ge 0})^2}\frac{q^{\frac{{\bf m}^T 3A[\nu]^{-1}  {\bf m}}{2}}}{(q;q)_{m_1}(q^3;q^3)_{m_2}},
\end{equation}
which is precisely the sum side of one of the mod $9$ conjectures of the second and third authors \cite{KanRus-idf} 
as found by Kur\c{s}ung\"{o}z~\cite{Kur-KR}.

The present work is the result of a discussion at the AMS Fall Sectional Meeting at Binghamton University involving Alejandro Ginory, the second author, and the fourth author. Namely, generalizing the shapes of \eqref{A2n-sum} and \eqref{stem-sum} (or \eqref{eqn:d43} and \eqref{eqn:d43inv}), do the sum sides of any other identities emerge when something similar is done to the graded dimensions of the principal subspaces of the basic modules for a twisted affine Lie algebra?
In this work, we show that, for matrices $A$ of type $\mathrm{A}$, $\mathrm{D}$, or $\mathrm{E}_6$, we do obtain $q$-series identities. In particular, we replace $A[\nu]$ with a suitable multiple (large enough to clear fractional entries) of $A[\nu]^{-1}$ and manipulate the denominator of each sum of the form~\eqref{twistedsum} in a predictable way: if the diagram automorphism has order $k=2$ or $k=3$, we replace instances of $(q;q)_n$ with $(q^k;q^k)_n$ and vice-versa. In all cases except when $A$ is of type $A_{2n-1}$ for $n \ge 2$ these identities come in pairs, producing two families of identities.

After experimentation by the first and fourth authors using Garvan's \texttt{qseries} Maple package~\cite{Ga}, new identities were found using a matrix $A$ of type $\mathrm{E}_6^{(2)}$:

\begin{align}
\sum_{i,j,k,\ell \ge 0}
\frac{q^{4i^2 + 12 ij + 8ik + 4i\ell + 12j^2 + 16jk + 8j\ell + 6k^2+6k\ell+2\ell^2}}
{\left(q^2;q^2\right)_i\left(q^2;q^2\right)_j\left(q;q\right)_k\left(q;q\right)_\ell}
&= \left(q^2,q^3,q^4,q^{10},q^{11},q^{12};q^{14}\right)^{-1}_\infty \label{eqn:Nandi_intro}, \\ 
\sum_{i,j,k,\ell \ge 0}
\frac{q^{2i^2 + 6 ij + 4ik + 2i\ell + 6j^2 + 8 jk + 4 j\ell + 3k^2+3k\ell+\ell^2}}
{\left(q^2;q^2\right)_i\left(q^2;q^2\right)_j\left(q;q\right)_k\left(q;q\right)_\ell} 
&=\left(q,q^2,q^4,q^6,q^8,q^9;q^{10}\right)^{-1}_\infty\label{eqn:mod10_intro}.
\end{align}

Notably, the product side of~\eqref{eqn:Nandi_intro} matches one of Nandi's identities. These were first conjectured by Nandi in his thesis~\cite{Nan-thesis} and later proved by Takigiku and Tsuchioka~\cite{TakTsu-nandi}. Remarkably, the (new) expression on the left side of~\eqref{eqn:Nandi_intro} above is a manifestly positive quadruple sum. (The sums used in Takigiku and Tsuchoika's proof are double sums, but are not manifestly positive.) Nandi's identities are connected to principal characters of level 4 standard modules of $\mathrm{A}_2^{(2)}$ (and also level 2 of $\mathrm{A}_{11}^{(2)}$).

In~\eqref{eqn:mod10_intro}, the left side is again a manifestly positive quadruple sum. In fact, the exponent of $q$ in the terms on the left side of~\eqref{eqn:Nandi_intro} is exactly twice the exponent of $q$ in the terms on the left side of~\eqref{eqn:mod10_intro}. The product side here is connected to level 4 of $\mathrm{D}_4^{(3)}$. The same relationship (of doubling the quadratic form) holds between Capparelli's identities \cite{Cap-1,KanRus-stair,Kur-Cap} which reside at level $3$ of $\mathrm{A}_2^{(2)}$ and Kur\c{s}ung\"oz's (multi)sum-to-product companions \cite{Kur-KR} to the conjectures of the second and third authors~\cite{KanRus-idf}  related to level $3$ of $\mathrm{D}_4^{(3)}$.

Sections 5 and 6 are dedicated to proofs of these identities and others in their families. We now outline our proof strategy. 

We begin by deducing $x,q$-relations (or, in the second case, $x,y,q$-relations) that sum sides equalling each respective product side are known to satisfy. After appropriately generalizing the quadruple sums, we then demonstrate relations that these generalized quadruple sums must satisfy. We finish our proofs by showing that that the desired relations follow from these known relations. In the case of the Nandi identities, this proof requires the use of a computer. This technique is similar to one that the second and third authors used in a previous paper~ \cite{KanRus-cylindric} (see also the work of Chern~\cite{Che-linked}). 

\subsection*{Acknowledgments}
We are indebted to S.\ Ole Warnaar for his most valuable comments on an earlier draft of this manuscript. 
S.K.\ acknowledges the support from the Collaboration Grant for Mathematicians \#636937 awarded by the Simons Foundation.

\section{Notation and preliminaries}
\label{sec:notations}
\subsection{Partitions}
We will write partitions in a weakly decreasing order. If $\lambda=\lambda_1+\lambda_2+\cdots +\lambda_j$ is a partition of $n$, 
we will say that the weight of $\lambda$ is $\wt(\lambda)=n$ and 
we will let $j=\ell(\lambda)$ be the length, or the number of parts, of $\lambda$. 
Each $\lambda_i$ will be called a part of $\lambda$.
Given a positive integer $i$,
we denote by $m_i(\lambda)$ the multiplicity of $i$ in $\lambda$. 
By a (contiguous) sub-partition of $\lambda$, we mean the subsequence
$\lambda_s+\cdots+\lambda_t$ of $\lambda$.
We say that $\lambda$ satisfies the difference condition $[d_1,d_2,\cdots,d_{j-1}]$ if $\lambda_i-\lambda_{i+1}=d_i$ for all $1\leq i\leq j-1$.

Let $\sC$ be any set of partitions. In the usual way, the $(x,q)$ generating function of $\sC$ is defined as
\begin{align}
f_{\sC}(x,q)=\sum_{\lambda\in\sC}x^{\ell(\lambda)}q^{\wt(\lambda)}.
\end{align}
Then, the $q$-generating function of $\sC$ is simply $f_{\sC}(1,q)$ (sometimes denoted just by $f_{\sC}(q)$).

As usual, we will let $\ZZ[[x,q]]$ be the ring of power series in variables $x,q$ with coefficients in $\ZZ$.
We will require a subset $\sS\subset \ZZ[[x,q]]$ of series $f\in \ZZ[[x,q]]$ such that $f(0,q)=f(x,0)=1$.

\subsection{\texorpdfstring{$q$-Series}{q-Series}}
We shall use standard notation regarding $q$-series. 
All of our series are formal, and issues of analytic convergence are disregarded.

For $n\in\ZZ_{\geq 0}\cup\{\infty\}$ we define:
\begin{align}
(a;q)_n=\prod_{0\leq t<n}(1-aq^{t}).
\end{align}
We will simply write $(a)_n$ when the base $q$ is understood.
To compress notation, we will write
\begin{align}
(a_1,a_2,\dots,a_r;q)_n=(a_1;q)_n(a_2;q)_n\cdots(a_r;q)_n.
\end{align}
We will also use modified theta functions:
\begin{align}
\theta(a;q)=(a;q)_\infty(q/a;q)_\infty,
\end{align}
and
\begin{align}
\theta(a_1,a_2,\dots,a_r;q)=\theta(a_1;q)\theta(a_2;q)\cdots(a_r;q).
\end{align}

\subsection{Vertex-algebraic background.}
We recall certain details relevant to this work from the vertex-algebraic constructions found in the works of Calinescu, Lepowsky, Milas, Penn, Webb, and the fourth author~\cite{CalLM4,CPS,PS1,PS2,PSW}. Suppose 
\begin{equation}
L = \mathbb{Z} \alpha_1 \oplus \cdots \oplus \mathbb{Z}\alpha_D
\end{equation}
 is a rank $D$ positive-definite even lattice, equipped with a nondegenerate $\mathbb{Z}$-bilinear form
\begin{equation}
\langle \cdot, \cdot \rangle: L \times L \rightarrow \mathbb{Z}
\end{equation}
 whose Gram matrix is either a Cartan matrix of type $\mathrm{A}$, $\mathrm{D}$, $\mathrm E$ or contains only non-negative entries. Consider the complexification of $L$ given by 
$$
\mathfrak{h} = L \otimes_{\mathbb{Z}} \mathbb{C}
$$
to which we extend $\langle \cdot, \cdot \rangle$.
 Let $\lambda_1, \dots, \lambda_D \in \mathfrak{h}$ be dual to the $\alpha_1,\dots, \alpha_D$ such that 
\begin{equation*}
\langle \alpha_i, \lambda_j \rangle = \delta_{i,j}
\end{equation*}
for $1 \le i,j \le D$.
 Suppose that $\nu:L \rightarrow L$ is an automorphism of order $k$ which permutes $\alpha_1,\dots, \alpha_D$. For $1 \le i \le D$ we have a $\nu$-orbit given by $\alpha_i$. Let $d$ be the number of disctinct $\nu$-orbits of the $\alpha_i$. Choose an $\alpha_{i_j}$ from each of the $d$ distinct orbits and let $\ell_j$ be the length of its orbit. For $1 \le j \le d$ we define 
$$
\beta_j = \frac{1}{k} \left( \alpha_{i_j} + \nu \alpha_{i_j} + \cdots + \nu^{k-1} \alpha_{i_j}\right)
$$
and
$$
\gamma_j = \frac{1}{k} \left( \lambda_{i_j} + \nu \lambda_{i_j} + \cdots + \nu^{k-1} \lambda_{i_j}\right).
$$
We define the matrix $A[\nu]$ by
$$
(A[\nu])_{i,j} = k\langle \beta_i, \beta_j \rangle
$$
for $1 \le i,j \le d$.

Let $V_L$ be the lattice vertex operator algebra constructed from $L$ (cf. the text of Lepowsky and Li~\cite{LL}). The automorphism $\nu$ can be extended to an automorphism $\hat{\nu}$ of $V_L^T$ of order $k$ or $2k$ (depending on $L$). Let $V_L^T$ be its $\hat{\nu}$-twisted modules as constructed by Lepowsky~\cite{L} (cf. also the work of Calinescu, Lepowsky, and Milas~\cite{CalLM4}). Let $W_L^T$ be the principal subspace of $V_L^T$. Calinescu, Lepowsky, and Milas~\cite{CalLM4} demonstrate that $V_L^T$ and $W_L^T$ are given compatible gradings by the operators $kL(0), \ell_1 \gamma_1(0),\dots, \ell_d \gamma_d(0)$ arising from the lattice construction of $V_L$ and $V_L^T$. In particular, $W_L^T$ is decomposed into finite-dimensional eigenspaces of these operators as
\begin{equation}
W_L^T = \coprod_{n\in \mathbb{Q}, m_1,\cdots,m_d \in \mathbb{N}} \left(W_L^T\right)_{(n,m_1,\dots,m_d)}
\end{equation} 
In particular, we define
$$\chi_{W_L^T}(x_1,x_2,\dots,x_d,q) =  q^{-\text{wt}1_T} tr\vert_{W_L^T} q^{kL(0)} x_1^{\ell_1\gamma_1(0)}\cdots x_d^{\ell_d\gamma_d(0)} $$
where $1_T$ is a vector of lowest conformal weight in $V_L^T$ (see Section 3 of ~\cite{PSW} for this notation) and $ q^{-\text{wt}1_T} $ is introduced to ensure that all powers of $q$ are nonnegative integers.
Following Penn, Webb, and the fourth author~\cite{PS1,PS2,PSW}, 
we have that for $1 \le i \le d$:
\begin{equation}\label{general-recursion}
\chi(x_1,\dots,x_d,q) = \chi(x_1,\dots,	q^{\frac{k}{\ell_i}}x_i,\dots,x_d,q) + x_iq^{k\frac{\langle \beta_i,\beta_i\rangle}{2}}\chi(x_1^{k\langle \beta_1,\beta_i\rangle},\dots,x_d^{k\langle \beta_d, \beta_i\rangle},q)
\end{equation}
which gives
\begin{equation}
\chi(x_1,\dots,x_d,q) = \sum_{{\bf m} \in \mathbb{N}^d} \frac{q^{\frac{{\bf m}^t A[\nu] {\bf m}}{2}}}{(q^{\frac{k}{\ell_1}};q^{{\frac{k}{\ell_1}}})\cdots (q^{\frac{k}{\ell_d}};q^{{\frac{k}{\ell_d}}})}x_1^{m_1}\cdots x_d^{m_d}.
\end{equation}
We call $\chi(x_1,\dots,x_d,q)$ the {\em multigraded dimension} of $W_L^T$ and call $\chi(1,\dots,1,q)$ simply the {\em graded dimension} of $W_L^T$. In this work we are primarily interested in the graded dimensions of $W_L^T$.

\section{Warmups}
\subsection{The Gordon--Andrews and  G\"ollnitz--Gordon--Andrews identities from $\mathrm{A}_{2n}^{(2)}$}
Here, we begin with the multigraded dimension of the principal subspace of the basic $\mathrm{A}_{2n}^{(2)}$ module found by Calinescu, Milas, and Penn~\cite{CalMPe}. In this case, the graded dimension is given by:
\begin{equation}
\sum_{{\bf m} \in (\mathbb{Z}_{\ge 0})^{n}}
\frac{q^{\frac{{\bf m}^T A[\nu]{\bf m}}{2}}}{{(q};q)_{m_1}\cdots (q;q)_{m_{n}}}
\end{equation}
where
\begin{center}
$
A[\nu] = \begin{bmatrix}
    2     & -1    & 0      & 0      & 0      & \ldots & \ldots & 0\\
    -1     &  2     & -1     & 0      & 0      & \ldots & \ldots & 0\\
    0      & -1     &  2     & -1     & 0      & \ldots & \ldots & 0\\
    \vdots & \vdots & \vdots & \ddots & \vdots & \ldots & \ldots & \vdots \\
    \vdots & \vdots & \vdots & \vdots & \ddots & \ldots & \ldots & \vdots \\
    \vdots & \vdots & \vdots & \vdots & \vdots & 2     & -1     & 0 \\
    \vdots & \vdots & \vdots & \vdots & \vdots & -1     & 2      & -1 \\
    \vdots & \vdots & \vdots & \vdots & \vdots & 0      & -1     & 1 \\
\end{bmatrix}$\\
\end{center}
is the Cartan matrix of the tadpole Dynkin diagram.
We now manipulate this sum as follows:
\begin{itemize}
\item Replace $A[\nu]$ with $2A[\nu]^{-1}$.
\item Replace each instance of $(q;q)_m$ with $(q^2;q^2)_m$.
\end{itemize}
Here we have that 
\begin{center}
$
2A[\nu]^{-1} = \begin{bmatrix}
    2     &2   & 2      & 2     &\hdots      & 2 & 2 & 2\\
    2     &  4     & 4     & 4      & \hdots     & 4 & 4 & 4\\
    2    & 4     &  6     & 6     & \hdots     & 6& 6 & 6\\
    \vdots & \vdots & \vdots & \vdots & \vdots & \vdots & \vdots & \vdots \\
    \vdots & \vdots & \vdots & \vdots & \vdots & \vdots & \vdots & \vdots \\
    2 & 4 & 6 & 8 & \hdots &2n-4     & 2n-4     & 2n-4 \\
    2 &4 &6 & 8 & \hdots & 2n-4     & 2n-2      & 2n-2 \\
    2 &4 & 6 & 8 & \hdots & 2n-4      & 2n-2     & 2n \\
\end{bmatrix}$\\
\end{center}
whose $(i,j)$-entry is $2\text{min}(i,j)$.
Using Garvan's {\tt qseries} package~\cite{Ga}, we obtain the following known identity (Corollary 1.5 (b) of Stembridge~\cite{S}, see also ~\cite{BIS} and Theorem 4.1 in ~\cite{War1}):
\begin{equation}
\sum_{{\bf m} \in (\mathbb{Z}_{\ge 0})^{n}}\frac{q^{\frac{{\bf m}^T 2A[\nu]^{-1}{\bf m}}{2}}}{{(q^2};q^2)_{m_1}\cdots (q^2;q^2)_{m_{n}}} =\frac{(-q;q^2)_\infty}{(q^2;q^2)_\infty} 
(q^{n+1},q^{n+3},q^{2n+4};\,q^{2n+4})_\infty.
\end{equation}
When $n$ is even, the product-side of this identity is the same as in the corresponding G\"ollnitz--Gordon--Andrews identity.
We note that 
$$
\frac{{\bf m}^T 2A[\nu]^{-1}{\bf m}}{2} = M_1^2 + M_2^2+ \cdots + M_n^2
$$
where we take 
\begin{equation}\label{big-M}
M_i = m_i + m_{i+1} + \cdots + m_n
\end{equation}
 for $1 \le i \le n$,
so that we can rewrite this identity as in the next theorem.
\begin{thm}
\begin{equation}
\sum_{{\bf m} \in (\mathbb{Z}_{\ge 0})^{n}}\frac{q^{M_1^2 + M_2^2 + \cdots + M_n^2}}{{(q^2};q^2)_{m_1}\cdots (q^2;q^2)_{m_{n}}} 
=\frac{(-q;q^2)_\infty}{(q^2;q^2)_\infty} (q^{n+1},q^{n+3},q^{2n+4};\,q^{2n+4})_\infty.
\end{equation}
\end{thm}
The sum side of this equation is the graded dimension for the principal subspace of the level $n$
$\mathrm{A}_2^{(2)}$ vacuum module (see the work of Calinescu, Penn, and the fourth author~\cite{CPS} and of Takenaka~\cite{Ta}).
Here, vacuum module refer to a standard module with highest weight $n\Lambda_0$.

Finally, we note that doubling the exponent of $q$ in the numerator of the sum side  yields one of the Andrews--Gordon identities, dilated by $q \mapsto q^2$:
\begin{thm} 
\begin{equation}
\sum_{{\bf m} \in (\mathbb{Z}_{\ge 0})^{n}}\frac{q^{\frac{{\bf m}^T 4A[\nu]^{-1}{\bf m}}{2}}}{{(q^2};q^2)_{m_1}\cdots (q^2;q^2)_{m_{n}}} = 
\dfrac{(q^{2n+2},q^{2n+4},q^{4n+6};\,q^{4n+6})_\infty}{(q^2;q^2)_\infty}.
\end{equation}
\end{thm}
Varying the linear term in the exponent of $q$ in the numerator of the left hand side of the equation yields the $q\mapsto q^2$ dilations of the remaining Andrews--Gordon identities~\cite{A}.

For a very different circle of ideas that connects the affine Lie algebra
$\mathrm{A}_2^{(2)}$ with the Gordon--Andrews identities, see 
Griffin, Ono and Warnaar's article~\cite{GriOnoWar}.

\subsection{The Bressoud identities and partner identities from $\mathrm{D}_n^{(2)}$}
Here, we examine the graded dimension of the principal subspace of the $\mathrm D_{n}^{(2)}$ basic module where $n \ge 3$. 
The graded dimension of the principal subspace of the basic module for $\mathrm D_n^{(2)}$ is given by:
\begin{equation}
\sum_{{\bf m} \in (\mathbb{Z}_{\ge 0})^{n-1}}\frac{q^{\frac{{\bf m}^T A[\nu] {\bf m}}{2}}}{{(q^2};q^2)_{m_1}\cdots (q^{2};q^{2})_{m_{n-2}}(q;q)_{m_{n-1}}}
\end{equation}
where $A[\nu]$ is the $(n-1) \times (n-1)$ matrix
\begin{align*}
A[\nu] = \begin{bmatrix}
    4      & -2     & 0      & 0      & 0      & \ldots & \ldots & 0\\
    -2     &  4     & -2     & 0      & 0      & \ldots & \ldots & 0\\
    0      & -2     &  4     & -2     & 0      & \ldots & \ldots & 0\\
    \vdots & \vdots & \vdots & \ddots & \vdots & \ldots & \ldots & \vdots \\
    \vdots & \vdots & \vdots & \vdots & \ddots & \ldots & \ldots & \vdots \\
    \vdots & \vdots & \vdots & \vdots & \vdots & 4      & -2     & 0 \\
    \vdots & \vdots & \vdots & \vdots & \vdots & -2     & 4      & -2 \\
    \vdots & \vdots & \vdots & \vdots & \vdots & 0      & -2     & 2 \\
\end{bmatrix}.
\end{align*}

We now manipulate this sum as follows:
\begin{itemize}
\item Replace $A[\nu]$ with $2A[\nu]^{-1}$.
\item Replace each instance of $(q^2;q^2)_m$ with $(q;q)_m$ and replace each instance of $(q;q)_m$ with $(q^2;q^2)_m$ 
\item Dilate with $q \mapsto q^2$.
\end{itemize}
Here, we have that
$$4A[\nu]^{-1} = \begin{bmatrix}
    2      & 2     & 2      & 2      & 2      & \ldots & \ldots & 2\\
    2    &  4     & 4    & 4      &4      & \ldots & \ldots & 4\\
    2      & 4    &  6     & 6     & 6      & \ldots & \ldots & 6\\
    \vdots & \vdots & \vdots & \vdots & \vdots & \ldots & \ldots & \vdots \\
    \vdots & \vdots & \vdots & \vdots & \vdots & \ldots & \ldots & \vdots \\
    2& 4& 6& 8  & 10 & \cdots      &   2n-4  & 2n-2\\
\end{bmatrix}\\
$$
In this case, our sum becomes:
\begin{equation}\label{Dn-sum}
\sum_{{\bf m} \in (\mathbb{Z}_{\ge 0})^{n-1}}\frac{q^{\frac{{\bf m}^T 4A[\nu]^{-1} {\bf m}}{2}}}{{(q^2};q^2)_{m_1}\cdots (q^2;q^2)_{m_{n-2}}(q^4;q^4)_{m_{n-1}}}
\end{equation}
where
\begin{equation}\label{Ms}
\frac{{\bf m}^T 4A[\nu]^{-1} {\bf m}}{2}= M_1^2 + M_2^2 + \cdots +M_{n-1}^2
\end{equation}
using the notation defined in (\ref{big-M}) above.
Using Garvan's {\tt qseries} package~\cite{Ga}, we obtain the following known identity (see equation (5.15) in ~\cite{War1}).
\begin{thm}
\begin{equation}
\sum_{{\bf m} \in (\mathbb{Z}_{\ge 0})^{n-1}}\frac{q^{M_1^2+M_2^2+\cdots+M_{n-1}^2}}{{(q^2};q^2)_{m_1}\cdots (q^2;q^2)_{m_{n-2}}(q^4;q^4)_{m_{n-1}}} 
= \frac{(q^n,q^{n+1},q^{2n+1};q^{2n+1})_\infty}{(q,q^3,q^4;q^4)_\infty}.
\end{equation}
\end{thm}

Finally, we note that if we modify the series (\ref{Dn-sum}) as follows: 
\begin{itemize}
\item Replace $4A[\nu]^{-1}$ with $8A[\nu]^{-1}$
\item Make the substitution $q\mapsto q^{1/2}.$
\end{itemize}
we obtain the sum side of one of Bressoud's mod $2n$ identities \cite{Br}:
\begin{equation}
\sum_{{\bf m} \in (\mathbb{Z}_{\ge 0})^{n-1}}\frac{q^{\frac{{\bf m}^T 4A[\nu]^{-1} {\bf m}}{2}}}{{(q};q)_{m_1}\cdots (q;q)_{m_{n-2}}(q^2;q^2)_{m_{n-1}}} = 
\frac{(q^n,q^n,q^{2n};q^{2n})_\infty}{(q)_\infty}.
\end{equation}
Of course, varying the linear terms in the exponent of $q$ in the sum yields the remaining of Bressoud's $\mod 2n$ identities.

\subsection{Identities from \texorpdfstring{$\mathrm{A}_{2n-1}^{(2)}$}{A\_{2n-1}\^{(2)}}}

Here, we repeat and adapt the process described above for the graded dimension of the principal subspace of the  \texorpdfstring{$\mathrm{A}_{2n-1}^{(2)}$}{A\_{2n-1}\^{(2)}} basic module. In particular, from \cite{PSW} the graded dimension of the principal subspace of the basic  \texorpdfstring{$\mathrm{A}_{2n-1}^{(2)}$}{A\_{2n-1}\^{(2)}}-module is
$$
\sum_{{\bf m} \in (\mathbb{Z}_{\ge 0})^n} \frac{q^{\frac{{\bf m}^T A[\nu] {\bf m}}{2}}}{(q;q)_{m_1} \cdots (q;q)_{m_{n-1}}(q^2; q^2)_{m_n}}
$$
where
\begin{center}
$
A[\nu] = \begin{bmatrix}
    2      & -1     & 0      & 0      & 0      & \ldots & \ldots & 0\\
    -1     &  2     & -1     & 0      & 0      & \ldots & \ldots & 0\\
    0      & -1     &  2     & -1     & 0      & \ldots & \ldots & 0\\
    \vdots & \vdots & \vdots & \ddots & \vdots & \ldots & \ldots & \vdots \\
    \vdots & \vdots & \vdots & \vdots & \ddots & \ldots & \ldots & \vdots \\
    \vdots & \vdots & \vdots & \vdots & \vdots & 2      & -1     & 0 \\
    \vdots & \vdots & \vdots & \vdots & \vdots & -1     & 2      & -2 \\
    \vdots & \vdots & \vdots & \vdots & \vdots & 0      & -2     & 4 
\end{bmatrix}$\\
\end{center}
We now manipulate the sum as follows:
\begin{itemize}
\item Replace $A[\nu]$ with $4A[\nu]^{-1}$.
\item Replace $(q^2;q^2)_m$ with $(q;q)_m$ and replace $(q;q)_m$ with $(q^2;q^2)_m$.
\item Dilate $q\mapsto q^2$ in the entire sum to avoid non-integer powers of $q$ when $n$ is odd.
\end{itemize}
Here, we have that 
\begin{center}
$
4A[\nu]^{-1} = \begin{bmatrix}
    4      & 4     & 4      & 4      & 4      & \ldots & \ldots & 2\\
    4    &  8     & 8     & 8      & 8     & \ldots & \ldots & 4\\
    4      & 8     &  12     & 12     & 12      & \ldots & \ldots & 6\\
    \vdots & \vdots & \vdots & \ddots & \vdots & \ldots & \ldots & \vdots \\
    \vdots & \vdots & \vdots & \vdots & \ddots & \ldots & \ldots & \vdots \\
    \vdots & \vdots & \vdots & \vdots & \vdots & 4n-8     & 4n-8     & 2n-4\\
    \vdots & \vdots & \vdots & \vdots & \vdots & 4n-8     & 4n-4      & 2n-2 \\
  2 & 4 & 6 & 8 & \hdots & 2n-4      & 2n-2    & n 
\end{bmatrix}$\\
\end{center}
i.e. 
\begin{equation}
(4A[\nu]^{-1})_{i,j} = \begin{cases}
4 \text{min}(i,j) & 1 \le i < n, 1 \le j < n \\
2j & i=n, 1 \le j < n\\
2i & j=n, 1 \le i <n\\
n & i=j=n.
\end{cases}
\end{equation}

Our new sum has the form:
\begin{equation}
\sum_{{\bf m} \in (\mathbb{Z}_{\ge 0})^n} \frac{q^{{\bf m}^T 8A[\nu]^{-1} {\bf m}/2}}{(q^4;q^4)_{m_1} \cdots (q^4;q^4)_{m_{n-1}}(q^2; q^2)_{m_n}}
\end{equation}

Using the {\tt qseries} package~\cite{Ga} we get:
\begin{align}
\sum_{{\bf m} \in (\mathbb{Z}_{\ge 0})^n} \frac{q^{{\bf m}^T 8A[\nu]^{-1} {\bf m}/2}}{(q^4;q^4)_{m_1} \cdots (q^4;q^4)_{m_{n-1}}(q^2; q^2)_{m_n}}
= \frac{(-q^{n},-q^{n+2},q^{2n+2};q^{2n+2})_\infty}{(q^4;q^4)_\infty}.
\label{eqn:A2n-1-conj}
\end{align}
We adopt the notation
\begin{equation}
N_i = \begin{cases} 2m_i + 2m_{i+1} + \cdots + 2m_{n-1} + m_n & 1 \le i \le n-1\\
m_n & i=n
\end{cases}
\end{equation}
so that the identity \ref{eqn:A2n-1-conj} can be rewritten as in the following Theorem.
\begin{thm}
$$\sum_{{\bf m} \in (\mathbb{Z}_{\ge 0})^n} \frac{q^{N_1^2 + N_2^2 + \cdots N_{n-1}^2 + N_n^2}}{(q^4;q^4)_{m_1} \cdots (q^4;q^4)_{m_{n-1}}(q^2; q^2)_{m_n}} 
= \frac{(-q^{n},-q^{n+2},q^{2n+2};q^{2n+2})_\infty}{(q^4;q^4)_\infty}.
$$
\end{thm}
\begin{proof}
This identity is proved easily. However, since we could not find it in the literature, we present a sketch of the proof.

Due to the definition of $N_i$, it is clear that they have the same parity.
Thus, we make two cases -- when each $N_i$ is even, we take $N_i=2n_i$ and when
each $N_i$ is odd, we take $N_i=2n_i+1$.
Once this is done, we map $q\mapsto q^{1/4}$ for convenience.
With a bit of algebraic manipulation, the sum then can be written as:
\begin{align}
    \sum_{n_1,\cdots,n_k\geq 0}&
    \dfrac{q^{n_1^2+\cdots+n_k^2}}
    {(q)_{n_1-n_2}\cdots (q)_{n_{k-1}-n_k}(q)_{n_k}(q^{1/2};q)_{n_k}}
    \nonumber\\
    &\quad 
    +\dfrac{q^{k/4}}{1-q^{1/2}}\sum_{n_1,\cdots,n_k\geq 0}
    \dfrac{q^{n_1^2+\cdots+n_k^2+n_1+\cdots+n_k}}
    {(q)_{n_1-n_2}\cdots (q)_{n_{k-1}-n_k}(q)_{n_k}(q^{3/2};q)_{n_k}}.
    \label{eqn:Aodd2sum}
\end{align}
We now handle each of these sums using the routine theory of Bailey pairs.

Recall that two sequences $\alpha_n, \beta_n$ ($n\geq 0$) form a Bailey pair
with respect to $a$ if for all $n\geq 0$ we have:
\begin{align}
    \beta_n=\sum_{r=0}^n\dfrac{\alpha_n}{(q)_{n-r}(aq)_{n+r}}.
\end{align}
We start with the pair $F(1)$ with respect to $a=1$ given by Slater \cite{Sla}:
\begin{align}
    \alpha_n=
    \begin{cases}
    1 & n=0 \\
    q^{n^2}(q^{-n/2}+q^{n/2}) & n>0 
    \end{cases},
    \quad 
    \beta_n=\dfrac{1}{(q)_n(q^{1/2};q)_n}.
\end{align}
Now, inserting this pair in \cite[Thm.\ 2]{And-mult} with $a=1$, we get:
\begin{align}
    \sum_{n_1,\cdots,n_k\geq 0}&
    \dfrac{q^{n_1^2+\cdots+n_k^2}}
    {(q)_{n_1-n_2}\cdots (q)_{n_{k-1}-n_k}(q)_{n_k}(q^{1/2};q)_{n_k}}
    =\dfrac{1}{(q)_\infty}\left(1+\sum_{n\geq 1}q^{(k+1)n^2}(q^{-n/2}+q^{n/2})\right)
\end{align}
In exactly the same way, if we start with Slater's pair $F(2)$ with respect to $a=q$:
\begin{align}
    \alpha_n=
   q^{n^2+n/2}\dfrac{1+q^{n+1/2}}{1+q^{1/2}}
    \quad 
    \beta_n=\dfrac{1}{(q)_n(q^{3/2};q)_n},
\end{align}
we arrive at:
\begin{align}
    \sum_{n_1,\cdots,n_k\geq 0}&
    \dfrac{q^{n_1^2+\cdots+n_k^2+n_1+\cdots+n_k}}
    {(q)_{n_1-n_2}\cdots (q)_{n_{k-1}-n_k}(q)_{n_k}(q^{1/2};q)_{n_k}}
    =\dfrac{1}{(q^2;q)_\infty}\sum_{n\geq 0}q^{(k+1)n^2+(2k+1)n/2}\dfrac{1+q^{n+1/2}}{1+q^{1/2}}.
\end{align}
Combining, we get that \eqref{eqn:Aodd2sum} equals:
\begin{align}
\dfrac{1}{(q)_\infty}\left(1+\sum_{n\geq 1}q^{(k+1)n^2}(q^{-n/2}+q^{n/2})+
\sum_{n\geq 0}q^{(k+1)n^2+(2k+1)n/2+k/4}(1+q^{n+1/2})\right).
\end{align}
We now reinstate $q^{1/4}\mapsto q$. 
After a few simple algebraic steps, the sums can be combined into a single sum.
\begin{align}
\dfrac{1}{(q^4;q^4)_\infty}\sum_{n\in\ZZ}q^{(k+1)n^2+n}.
\end{align}
A straight-forward application of the Jacobi Triple Product identity \cite[Thm.\ 2.8]{And-book} now gives
the required product.
\end{proof}

\subsection{Identities from \texorpdfstring{$\mathrm{D}_4^{(3)}$}{D\_4\^(3)}}
This process can also be used to rediscover double sums for Capparelli's identities and the mod 9 conjectured identities of the second and third authors~\cite{KanRus-idf}, in the form given by Kur\c{s}ung\"{o}z~\cite{Kur-KR}.

Here, we repeat and adapt the process described above for the character of the principal subspace of the $\mathrm{D}_4^{(3)}$ basic module.

The graded dimension of the principal subspace of the basic $\mathrm{D}_4^{(3)}$-module is
\begin{equation}
\sum_{{\bf m} \in (\mathbb{Z}_{\ge 0})^2} \frac{q^{\frac{{\bf m}^T A[\nu] {\bf m}}{2}}}{(q^3;q^3)_{m_1} (q;q)_{m_{2}}}
\end{equation}
where 
$$
A[\nu] = \begin{bmatrix}
6 & -3\\
-3 & 2
\end{bmatrix}.$$
We manipulate the sum as follows:
\begin{itemize}
\item Replace $A[\nu]$ with $3A[\nu]^{-1}$.
\item Replace $(q^3;q^3)_m$ with $(q;q)_m$ and replace $(q;q)_m$ with $(q^3;q^3)_m$.
\end{itemize}
\begin{equation}
\sum_{{\bf m} \in (\mathbb{Z}_{\ge 0})^4} \frac{q^{\frac{{\bf m}^T 3A[\nu]^{-1} {\bf m}}{2}}}{(q;q)_{m_1} (q^3;q^3)_{m_{2}}}
\end{equation}
Here, we have that
$$
3A[\nu]^{-1} = \begin{bmatrix}
2 & 3\\
3 & 6
\end{bmatrix}
$$
and that
$$
\frac{{\bf m}^T 3A[\nu]^{-1} {\bf m}}{2} = m_1^2 + 3m_1 m_2 + 3m_2^2.
$$
The sum under consideration is thus:
$$
\sum_{{\bf m} \in \left(\mathbb{Z}_{\ge 0}\right)^2} \frac{q^{ m_1^2 + 3m_1m_2 + 3m_2^2}}{(q;q)_{m_1}(q^3;q^3)_{m_2}},
$$
which was originally shown  by Kur\c{s}ung\"{o}z~\cite{Kur-KR} to be an analytic sum side for the identity $I_1$ of the second and third authors~\cite{KanRus-idf}.  Varying the linear terms produces further conjectural results from those papers (see also the work of Hickerson as given by Konenkov~\cite{Kon}, which in turn is inspired by the paper of Uncu and Zudilin \cite{UncZud}):
\begin{conj}
\begin{align*}
\sum_{{\bf m} \in \left(\mathbb{Z}_{\ge 0}\right)^2} \frac{q^{ m_1^2 + 3m_1m_2 + 3m_2^2}}{(q;q)_{m_1}(q^3;q^3)_{m_2}} &= 
\dfrac{1}{\theta(q,q^3;\,q^9)},
\\
\sum_{{\bf m} \in \left(\mathbb{Z}_{\ge 0}\right)^2} \frac{q^{ m_1^2 + 3m_1m_2 + 3m_2^2+m_1+3m_2}}{(q;q)_{m_1}(q^3;q^3)_{m_2}} &= 
\dfrac{1}{\theta(q^2,q^3;\,q^9)},
\\
\sum_{{\bf m} \in \left(\mathbb{Z}_{\ge 0}\right)^2} \frac{q^{ m_1^2 + 3m_1m_2 + 3m_2^2+2m_1+3m_2}}{(q;q)_{m_1}(q^3;q^3)_{m_2}} &= 
\dfrac{1}{\theta(q^3,q^4;\,q^9)}.
\end{align*}
\end{conj}
We note here that these identities are related to the principal characters of level $3$ standard modules for the twisted affine Lie algebra $\mathrm{D}_4^{(3)}$. We also note that doubling the exponent of $q$ in the numerator of the above analytic sum side  (i.e. using $6A[\nu]^{-1}$ in place of $3A[\nu]^{-1}$) yields the following double sum version of Capparelli's identity, as deduced by the second and third authors~\cite{KanRus-stair} and independently by Kur\c{s}ung\"{o}z~\cite{Kur-Cap}:
\begin{thm}
$$
\sum_{{\bf m} \in \left(\mathbb{Z}_{\ge 0}\right)^2} \frac{q^{ 2m_1^2 + 6m_1m_2 + 6m_2^2}}{(q;q)_{m_1}(q^3;q^3)_{m_2}} = 
\frac{1}{\theta(q,q^3;\,q^{12})}.
$$
\end{thm}

\section{Identities from \texorpdfstring{$\mathrm{E}_6^{(2)}$}{E\_6\^(2)}}
Here, we repeat the process described above for the character of the principal subspace of the $\mathrm{E}_6^{(2)}$ basic module. In this case, we have 
 \begin{center}
$
A[\nu] = \begin{bmatrix}
    2 & -1 & 0 & 0\\
    -1 & 2 & -2 & 0\\
    0 & -2 & 4 & -2\\
    0 & 0 & -2 & 4
\end{bmatrix}$
\end{center}
and the principal subspace of the basic $\mathrm{E}_6^{(2)}$-module has graded dimension given by:
\begin{equation}
\sum_{{\bf m} \in (\mathbb{Z}_{\ge 0})^4} \frac{q^{\frac{{\bf m}^T A[\nu] {\bf m}}{2}}}{(q;q)_{m_1} (q;q)_{m_{2}}(q^2; q^2)_{m_3}(q^2;q^2)_{m_4}}.
\end{equation}
We manipulate the sum as follows:
\begin{itemize}
\item Replace $A[\nu]$ with $2A[\nu]^{-1}$.
\item Replace $(q^2;q^2)_m$ with $(q;q)_m$ and replace $(q;q)_m$ with $(q^2;q^2)_m$
\end{itemize}
which now gives us the sum:
\begin{equation}
\sum_{{\bf m} \in (\mathbb{Z}_{\ge 0})^4} \frac{q^{\frac{{\bf m}^T 2A[\nu]^{-1} {\bf m}}{2}}}{(q^2;q^2)_{m_1} (q^2;q^2)_{m_{2}}(q; q)_{m_3}(q;q)_{m_4}}.
\end{equation}
We have that
$$
2A[\nu]^{-1}= \begin{bmatrix}
4 & 6 & 4  & 2\\
6 & 12 & 8 & 4\\
4 & 8 & 6 & 3\\
2 & 4 & 3 & 2
\end{bmatrix}
$$
so that
$$
\frac{{\bf m}^T 2A[\nu]^{-1} {\bf m}}{2} = 2m_1^2+6m_1m_2+4m_1m_3+2m_1m_4+6m_2^2+8m_2m_3+4m_2m_4+3m_3^2+3m_3m_4+m_4^2.
$$
Thus the sum under consideration is:
$$
\sum_{{\bf m} \in \left(\mathbb{Z}_{\ge 0}\right)^4} \frac{q^{ 2m_1^2+6m_1m_2+4m_1m_3+2m_1m_4+6m_2^2+8m_2m_3+4m_2m_4+3m_3^2+3m_3m_4+m_4^2}}{(q^2;q^2)_{m_1}(q^2;q^2)_{m_2}(q;q)_{m_3}(q;q)_{m_4}}.
$$
Using Garvan's package~\cite{Ga}, we obtain the following identity, which we will subsequently prove:
\begin{align*}
\sum_{{\bf m} \in \left(\mathbb{Z}_{\ge 0}\right)^4} \frac{q^{ 2m_1^2+6m_1m_2+4m_1m_3+2m_1m_4+6m_2^2+8m_2m_3+4m_2m_4+3m_3^2+3m_3m_4+m_4^2}}{(q^2;q^2)_{m_1}(q^2;q^2)_{m_2}(q;q)_{m_3}(q;q)_{m_4}}
=\dfrac{1}{\theta(q;q^5)\theta(q^2;q^{10})}.
\end{align*}
The product-side of this expression is the principal character of a level 4 standard $\mathrm{D}_4^{(3)}$-module. 
Adding appopriate linear terms to the exponent of $q$ in the numerator of the analytic sum-side of the previous conjecture yields the conjectural identities related to the remaining level $4$ standard modules of $\mathrm{D}_4^{(3)}$. 
This family of identities along with proofs is presented in Section \ref{sec:mod10} below.

Additionally, doubling the the exponent of $q$ used in the above analytic sum side, (i.e., using $4A[\nu]^{-1}$ in place of $2A[\nu]^{-1}$) yields the following conjecture, where the product side is exactly that of one of Nandi's identities related to level $4$ principal characters of $\mathrm{A}_2^{(2)}$ \cite{Nan-thesis}. 
\begin{align}
\label{eqn:nandi-quadsum-conj}
\sum_{{\bf m} \in \left(\mathbb{Z}_{\ge 0}\right)^4} &\frac{q^{ 4m_1^2+12m_1m_2+8m_1m_3+4m_1m_4+12m_2^2+16m_2m_3+8m_2m_4+6m_3^2+6m_3m_4+2m_4^2}}{(q^2;q^2)_{m_1}(q^2;q^2)_{m_2}(q;q)_{m_3}(q;q)_{m_4}}
&=
\dfrac{1}{\theta(q^2,q^3,q^4;\,q^{14})}.
\end{align}

\section{Quadruple sums for Nandi's identities}
In this section, we will state and prove the quadruple sum representations for Nandi's identities,
including \eqref{eqn:nandi-quadsum-conj}.

\subsection{The identities}
Nandi conjectured the following partition identities in his thesis~\cite{Nan-thesis}. These identities were
then proved by Takigiku and Tsuchioka~\cite{TakTsu-nandi}. We now recall these identities.

Let $\sN$ be the set of partitions $\lambda$ that satisfy both the following conditions:
\begin{enumerate}
\item No sub-partition of $\lambda$ satisfies the difference conditions $[1], [0,0], [0,2], [2,0]$ or $[0,3]$.
\item No sub-partition with an odd weight satisfies the difference conditions $[3,0], [0,4], [4,0]$ or $[3,2*,3,0]$ (where $2*$ indicates zero or more occurrences of $2$).
\end{enumerate}

We further define the following sets:
\begin{align}
\sN_1&=\{ \lambda\in \sN\,\vert\, m_1(\lambda)=0\},\\
\sN_2&=\{ \lambda\in \sN\,\vert\, m_1(\lambda), m_2(\lambda), m_3(\lambda)\leq 1\},\\
\sN_3&=\left\lbrace 
\lambda\in \sN\,\left\vert\, 
\begin{matrix}
m_1(\lambda)=m_3(\lambda)=0,  m_2(\lambda)\leq 1,\\
\lambda\,\,\mathrm{has\,\,no\,\,subpartition\,\,of\,\,the\,\,form\,\,} \\
(2k+3)\,\,+\,\,2k+(2k-2)+\cdots+4+2\,\, \mathrm{with}\,\, k\geq 1.
\end{matrix}
\right.
\right\rbrace
\end{align}

Using this notation, we can state the following theorem.
\begin{thm}[Conjectured by Nandi~\cite{Nan-thesis}, proved by Takigiku and Tsuchioka~\cite{TakTsu-nandi}]
For any $n\geq 0$, we have the following three identities.
\begin{enumerate}
	\item The number of partitions of $n$ belonging to $\sN_1$ is the same as the
	number of partitions of $n$ into parts congruent to $\pm 2$, 
	$\pm 3$  or $\pm 4$ modulo $14$.
	\item The number of partitions of $n$ belonging to $\sN_2$ is the same as the
	number of partitions of $n$ into parts congruent to $\pm 1$, 
	$\pm 4$  or $\pm 6$ modulo $14$.
	\item The number of partitions of $n$ belonging to $\sN_3$ is the same as the
	number of partitions of $n$ into parts congruent to $\pm 2$, 
	$\pm 5$  or $\pm 6$ modulo $14$.
\end{enumerate}
\end{thm}

\subsection{Difference equations}
Takigiku and Tsuchioka's remarkable proof of these identities~\cite{TakTsu-nandi} relies on
a certain system of difference equations satisfied by the generating functions of $\sN_1$, $\sN_2$ and $\sN_3$.

Consider the following system:
\begin{align}
\begin{bmatrix}
F_0(x,q)\\
F_1(x,q)\\
F_2(x,q)\\
F_3(x,q)\\
F_4(x,q)\\
F_5(x,q)\\
F_7(x,q)
\end{bmatrix}
=
\begin{bmatrix} 
1&x{q}^{2}&{x}^{2}{q}^{4}&xq&{x}^{2}{q
}^{2}&0&0\\ 
0&x{q}^{2}&0&0&0&1&0\\ 
0&0&0&0&0&0&1\\ 
0&x{q}^{2}&0&xq&0&1&0\\ 
0&0&0&0&x{q}^{2}&0&1\\ 
1&x{q}^{2}&{x}^{2}{q}^{4}&xq&0&0&0\\
1&x{q}^{2}&{x}^{2}{q}^{4}&0&0&0&0
\end{bmatrix}
\begin{bmatrix}
F_0(xq^2,q)\\
F_1(xq^2,q)\\
F_2(xq^2,q)\\
F_3(xq^2,q)\\
F_4(xq^2,q)\\
F_5(xq^2,q)\\
F_7(xq^2,q)
\end{bmatrix}
\label{eqn:Ndiffeq}
\end{align}
where each $F_i(x,q)\in\ZZ[[x,q]]$ for $i=0,\dots,5$ or $i=7$ is a 
generating function of certain set of partitions, say $\sC_i$.
Then, Takigiku and Tsuchioka prove that the generating functions $f_{\sN_i}(x,q)$ with $i=1,2,3$ satisfy this system
if we take:
\begin{align}
F_7(x,q)&=f_{\sN_1}(x,q),\\
F_3(x,q)&=f_{\sN_2}(x,q),\\
F_4(x,q)&=f_{\sN_3}(x,q).
\end{align}

From here, we may use the modified Murray--Miller algorithm 
to obtain an $x,q$-difference equation satisfied by
each of the $F_i$s with $i=0,\dots,5$ or $i=7$.
We shall follow Takigiku and Tsuchioka's exposition of this algorithm~\cite{TakTsu-nandi}; see also the expositions of Andrews~\cite[Ch.\ 8]{And-book} and Chern~\cite{Che-linked} for more examples.
In each case, it is easy to see that the resulting difference equation has a unique
solution in the set $\sS$.

\begin{prop}
\label{prop:diffeq}
The power series $F_1(x,q)$ is the unique solution in $\sS$ of:
\begin{align}
    0&=F_1(x,q)  
    +(-q^5x-q^4x-q^2x-1)F_1(xq^2,q) 
    +q^3x(q^8x+q^6x+q^2+q-1)F_1(xq^4,q)
    \nonumber\\
    &+x^2q^8(q^8x+q^6x-q^3+q-1)F_1(xq^6,q)
    -q^{16}x^3(q^{11}x+q^9x+q^8x-q^3-q-1)F_1(xq^8,q)
    \nonumber\\
    &+x^3q^{19}(q^{18}x^2-q^{10}x-q^8x+1)F_1(xq^{10},q)
    \label{eqn:recF1}.
\end{align}
The power series $F_5(x,q)$ is the unique solution in $\sS$ of:
\begin{align}
    0 =&  
    F_5(x,q)
    +(-q^4x-q^3x-q^2x-1)F_5(xq^2,q)
    \nonumber\\
    &+xq(q^8x+q^7x+q^6x-q^3x+q^3+q^2-1)F_5(xq^4,q)
    \nonumber\\
    &-x^2q^4(q^{11}x-q^8x-q^7x-q^6x+q^5-q^3+1)F_5(xq^6,q)
    \nonumber\\
    &-q^{11}x^3(q^{10}x+q^9x+q^8x-q^2-q-1)F_5(xq^8,q)
    \nonumber\\
    &+q^{13}x^3(q^{18}x^2-q^{10}x-q^8x+1)F_5(xq^{10},q)
    \label{eqn:recF5}.
\end{align}
The power series $F_7=f_{\sN_1}$ is the unique solution in $\sS$ of:
\begin{align}
    0 =&  
    F_7(x,q)
	+(-q^4x-q^3x-q^2x-1)F_7(xq^2,q)
	+(q^5x+q^4x+q^3x-x+1)q^4xF_7(xq^4,q)
	\nonumber\\
	&
	-x^2q^6(q^9x-q^6x-q^5x-q^4x+1)F_7(xq^6,q)
	-x^3q^{13}(q^8x+q^7x+q^6x-q^2-q-1)F_7(xq^8,q)
	\nonumber\\
	&
	+x^3q^{17}(q^{14}x^2-q^8x-q^6x+1)F_7(xq^{10},q)
    \label{eqn:recF7}.
\end{align}
(For this last equation, see also equation (C.1) of Takigiku and Tsuchioka~\cite{TakTsu-nandi}.)
\end{prop}

Once unique solutions to $F_1, F_5, F_7 \in \sS$ have been found,
$F_0, F_2, F_3, F_4$ are uniquely determined due to the system \eqref{eqn:Ndiffeq} as follows.
\begin{prop} 
\label{prop:intermsof157}
We have:
\begin{align}
F_2(x,q)&=F_7(xq^2,q), \label{eqn:F2}\\
F_3(x,q)&=F_1(x,q) + F_5(x,q)-F_7(x,q),\label{eqn:F3}\\
F_0(x,q)&=F_7(xq^{-2},q)-xF_1(x,q)-x^2F_2(x,q),\label{eqn:F0}\\
F_4(x,q)&=x^{-2}q^{2}F_0(xq^{-2},q)-x^{-2}q^{2}F_5(xq^{-2},q).\label{eqn:F4}
\end{align}
\end{prop}
\begin{proof}
\eqref{eqn:F3} follows by comparing the recurrences for $F_1(x,q)$, $F_5(x,q)$ and $F_7(x,q)$;
\eqref{eqn:F0} follows by solving the recurrence for $F_7(x,q)$ in \eqref{eqn:Ndiffeq}; and
\eqref{eqn:F4} by comparing the recurrences for $F_0(x,q)$ and $F_5(x,q)$.
\end{proof}

\subsection{Proofs of our sum sides}
To enable us to deduce $x,q$-recurrences, we modify the quadruple sum in~\eqref{eqn:Nandi_intro} by inserting in the variable $x$, along with including linear terms in the exponent of $q$. To this end, we define:
\begin{align}
S_{A,B,C,D}(x,q) = \sum_{i,j,k,\ell \ge 0}
\frac{x^{2i+3j+2k+\ell} q^{4i^2 + 12 ij + 8ik + 4i\ell + 12j^2 + 16jk + 8j\ell + 6k^2+6k\ell+2\ell^2+Ai+Bj+Ck+D\ell}}
{\left(q^2;q^2\right)_i\left(q^2;q^2\right)_j\left(q;q\right)_k\left(q;q\right)_\ell}.
\end{align}

We will typically suppress the $x$ and $q$ arguments when they are clear.

Our main theorem of this section is the following.
\begin{thm} 
\label{thm:nandimain}
We have:
\begin{align}
    F_1(x,q) &= S_{2,2,1,0}(x,q), \label{eqn:SF1}\\
    F_5(x,q) &= S_{0,-2,-2,-1}(x,q), \label{eqn:SF5}\\
    F_7(x,q) &= S_{0,0,0,0}(x,q). \label{eqn:SF7}
\end{align}
\end{thm}
The rest of this section is devoted to the proof of this theorem.
We begin by deducing certain fundamental relations satisfied by $S_{A,B,C,D}$.
We clearly have:
\begin{align}
S_{A,B,C,D}(xq,q)=S_{A+2,B+3,C+2,D+1}(x,q).
\label{eqn:Sshift}
\end{align}
Additionally, we have:
\begin{alignat}{3}
\widehat{n_1}(A,B,C,D):&\quad S_{A,B,C,D}-S_{A+2,B,C,D}-x^2q^{4+A}S_{A+8,B+12,C+8,D+4} &&=0, \label{eqn:reln1hat}\\
n_2(A,B,C,D):&\quad S_{A,B,C,D}-S_{A,B+2,C,D}-x^3q^{12+B}S_{A+12,B+24,C+16,D+8} &&=0, \\
\widehat{n_3}(A,B,C,D):&\quad S_{A,B,C,D}-S_{A,B,C+1,D}-x^2q^{6+C}S_{A+8,B+16,C+12,D+6}  &&=0, \label{eqn:reln3hat}\\
\widehat{n_4}(A,B,C,D):&\quad S_{A,B,C,D}-S_{A,B,C,D+1}-xq^{2+D}S_{A+4,B+8,C+6,D+4}  &&=0. \label{eqn:reln4hat}
\end{alignat}
(compare with~\eqref{general-recursion}). We will be modifying $\widehat{n_1}, \widehat{n_3}, \widehat{n_4}$ shortly to our final relations
$n_1, n_3, n_4$, respectively.
To prove \eqref{eqn:reln1hat}:
\begin{align*}
&S_{A,B,C,D}-S_{A+2,B,C,D} \\
&=\sum_{i,j,k,\ell \ge 0}
\frac{x^{2i+3j+2k+\ell} q^{4i^2 + 12 ij + 8ik + 4i\ell + 12j^2 + 16jk + 8j\ell + 6k^2+6k\ell+2\ell^2+Ai+Bj+Ck+D\ell} \left(1-q^{2i}\right)}
{\left(q^2;q^2\right)_i\left(q^2;q^2\right)_j\left(q;q\right)_k\left(q;q\right)_\ell} \\
&=
\sum_{i,j,k,\ell \ge 0}
\frac{x^{2i+3j+2k+\ell} q^{4i^2 + 12 ij + 8ik + 4i\ell + 12j^2 + 16jk + 8j\ell + 6k^2+6k\ell+2\ell^2+Ai+Bj+Ck+D\ell} }
{\left(q^2;q^2\right)_{i-1}\left(q^2;q^2\right)_j\left(q;q\right)_k\left(q;q\right)_\ell} \\
&=
\sum_{\ii,j,k,\ell \ge 0}
\frac{x^{2\left(\ii+1\right)+3j+2k+\ell} q^{4\left(\ii+1\right)^2 + 12 \left(\ii+1\right)j + 8\left(\ii+1\right)k + 4\left(\ii+1\right)\ell + 12j^2 + 16jk + 8j\ell + 6k^2+6k\ell+2\ell^2+A\left(\ii+1\right)+Bj+Ck+D\ell} }
{\left(q^2;q^2\right)_{\ii}\left(q^2;q^2\right)_j\left(q;q\right)_k\left(q;q\right)_\ell} \\
&= x^2q^{4+A} 
\sum_{\ii,j,k,\ell \ge 0}
\frac{x^{2\ii+3j+2k+\ell} q^{4\ii^2 + 12\ii j + 8\ii k + 4\ii \ell + 12j^2 + 16jk + 8j\ell + 6k^2+6k\ell+2\ell^2+(A+8)\ii+(B+12)j+(C+8)k+(D+4)\ell} }
{\left(q^2;q^2\right)_{\ii}\left(q^2;q^2\right)_j\left(q;q\right)_k\left(q;q\right)_\ell} \\
&=x^2q^{4+A} S_{A+8,B+12,C+8,D+4}
\end{align*}
Note the reindexing $\ii=i+1$ in the middle of the above calculation. Proofs for the other three fundamental relations are similar, involving multiplication by $\left(1-q^{2j}\right)$, $\left(1-q^k\right)$, and $\left(1-q^\ell\right)$, respectively.

Because of the structure of what we will actually need to prove, we will prefer to work with modified versions of 
relations $\widehat{n_1}$, $\widehat{n_3}$ and $\widehat{n_4}$.
To modify $\widehat{n_1}(A,B,C,D)$, we combine three copies of $\widehat{n_1}$ in  the following way:
\begin{align*}
\widehat{n_1}&(A,B,C,D)+\widehat{n_1}(A+2,B,C,D)-x^2q^{6+A}\widehat{n_1}(A+8,B+12,C+8,D+4) \\
&=S_{A,B,C,D}-S_{A+2,B,C,D}-x^2q^{4+A}S_{A+8,B+12,C+8,D+4} \\ 
&\quad+S_{A+2,B,C,D}-S_{A+4,B,C,D}-x^2q^{6+A}S_{A+10,B+12,C+8,D+4}\\
&\quad-x^2q^{6+A}\left(S_{A+8,B+12,C+8,D+4}-S_{A+10,B+12,C+8,D+4}-x^2q^{12+A}S_{A+16,B+24,C+16,D+8}\right) \\
&= S_{A,B,C,D}-x^2q^{4+A}S_{A+8,B+12,C+8,D+4}-S_{A+4,B,C,D}-x^2q^{6+A}S_{A+8,B+12,C+8,D+4}\\
&\quad+x^4q^{18+2A}S_{A+16,B+24,C+16,D+8} = 0
\end{align*}
We call this final relation $n_1(A,B,C,D)$. Similarly, to find our ${n_3}(A,B,C,D)$ relation, we combine 
$$\widehat{n_3}(A,B,C,D) + \widehat{n_3}(A,B,C+1,D) - x^2 q^{7+C}  \widehat{n_3}(A+8,B+16,C+12,D+6),$$
and to find our ${n_4}(A,B,C,D)$ relation, we combine 
$$\widehat{n_4}(A,B,C,D) + \widehat{n_4}(A,B,C,D+1) - xq^{3+D} \widehat{n_4}(A+4,B+8,C+6,D+4).$$

We record the final list of relations in the following proposition.
\begin{prop}
\label{prop:nandirels}
The objects $S_{A,B,C,D}$ satisfy the following set of relations.
\begin{align}
n_1(A,B,C,D):&\quad S_{A,B,C,D}-x^2q^{4+A}(1+q^2)S_{A+8,B+12,C+8,D+4} -S_{A+4,B,C,D}\nonumber\\
&\quad+x^4q^{18+2A}S_{A+16,B+24,C+16,D+8} =0, \label{eqn:reln1}\\
n_2(A,B,C,D):&\quad S_{A,B,C,D}-S_{A,B+2,C,D}-x^3q^{12+B}S_{A+12,B+24,C+16,D+8} =0, \label{eqn:reln2}\\
{n_3}(A,B,C,D):&\quad S_{A,B,C,D}-x^2q^{6+C}S_{A+8,B+16,C+12,D+6}-S_{A,B,C+2,D} \nonumber\\ 
&\quad-x^2q^{7+C}S_{A+8,B+16,C+12,D+6}+x^4q^{25+2C}S_{A+16,B+32,C+24,D+12} =0, \label{eqn:reln3}\\
{n_4}(A,B,C,D):&\quad S_{A,B,C,D}-xq^{2+D}S_{A+4,B+8,C+6,D+4}-S_{A,B,C,D+2} \nonumber\\ 
&\quad-xq^{3+D}S_{A+4,B+8,C+6,D+4}+x^2q^{9+2D}S_{A+8,B+16,C+12,D+8} = 0.
\label{eqn:reln4}
\end{align}
\end{prop}

We now prove Theorem \ref{thm:nandimain}.
\begin{proof}[Proof of Theorem \ref{thm:nandimain}]
The main idea is to show that in each of the equations \eqref{eqn:SF7}, \eqref{eqn:SF1} and \eqref{eqn:SF5},
both sides are the (unique) solutions in $\sS$ to the corresponding 
difference equations given in Proposition \ref{prop:diffeq}.

It is clear that $S_{0,0,0,0}(x,q)$, $S_{2,2,1,0}(x,q)$, $S_{0,-2,-2,-1}(x,q)$ belong to $\sS$.

Establishing the recurrence \eqref{eqn:recF1} for $S_{2,2,1,0}(x,q)$ amounts to proving:
\begin{align}
    0&=S_{2,2, 1, 0} 
    +(-q^5x-q^4x-q^2x-1)S_{6, 8, 5, 2}
    +q^3x(q^8x+q^6x+q^2+q-1)S_{10, 14, 9, 4}
    \nonumber\\
    &\quad+x^2q^8(q^8x+q^6x-q^3+q-1)S_{14, 20, 13, 6}
    -q^{16}x^3(q^{11}x+q^9x+q^8x-q^3-q-1)S_{18, 26, 17, 8}
    \nonumber\\
    &\quad+x^3q^{19}(q^{18}x^2-q^{10}x-q^8x+1)S_{22, 32, 21, 10}
    \label{eqn:recS_F1}.
\end{align}
The file \texttt{F1.txt} provides this relation as a 
(huge!) linear combination of the fundamental relations in Proposition \ref{prop:nandirels}.

Establishing the recurrence \eqref{eqn:recF5} for $S_{0,-2,-2,-1}(x,q)$ amounts to proving:
\begin{align}
    0 =&  
    S_{0, -2, -2, -1}
    +(-q^4x-q^3x-q^2x-1)S_{4, 4, 2, 1}
    +xq(q^8x+q^7x+q^6x-q^3x+q^3+q^2-1)S_{8, 10, 6, 3}
    \nonumber\\
    &-x^2q^4(q^{11}x-q^8x-q^7x-q^6x+q^5-q^3+1)S_{12, 16, 10, 5}
    \nonumber\\
    &-q^{11}x^3(q^{10}x+q^9x+q^8x-q^2-q-1)S_{16, 22, 14, 7}
    \nonumber\\
    &+q^{13}x^3(q^{18}x^2-q^{10}x-q^8x+1)S_{20, 28, 18, 9}
    \label{eqn:recS_F5}.
\end{align}
The file \texttt{F5.txt} provides this relation as a 
linear combination of the fundamental relations in Proposition \ref{prop:nandirels}.

Using \eqref{eqn:Sshift}, establishing the recurrence \eqref{eqn:recF7} for $S_{0,0,0,0}(x,q)$ amounts to proving:
\begin{align}
    0 =&  
    S_{0, 0, 0, 0}
	+(-q^4x-q^3x-q^2x-1)S_{2,3,2,1}
	+q^4x(q^5x+q^4x+q^3x-x+1)S_{8, 12, 8, 4}
	\nonumber\\
	&
	-x^2q^6(q^9x-q^6x-q^5x-q^4x+1)S_{12, 18, 12, 6}
	-x^3q^{13}(q^8x+q^7x+q^6x-q^2-q-1)S_{16, 24, 16, 8}
	\nonumber\\
	&
	+x^3q^{17}(q^{14}x^2-q^8x-q^6x+1)S_{20, 30, 20, 10}
    \label{eqn:recS_F7}.
\end{align}
The file \texttt{F7.txt} provides this relation as a 
linear combination of the fundamental relations in Proposition \ref{prop:nandirels}.

In each case above, the linear combinations are very large and it is impossible to check 
them by hand.
One may simply import the files above in a computer algebra system, systematically replace
all symbols $n_1, n_2, n_3, n_4$ by the corresponding relations, simplify the answer
and finally check that the required relations are obtained. 
This can be implemented very easily and we provide
the required programs. See Appendix \ref{sec:proofverification} for details.
\end{proof}

Now that Theorem \ref{thm:nandimain} is proved, we may use Proposition \ref{prop:intermsof157}
to finally arrive at sum sides for all of Nandi's identities.
\begin{thm}
\label{thm:nandisums}
We have:
\begin{align}
f_{\sN_1}(x,q)&=F_7(x,q)=S_{0,0,0,0}(x,q),\\
f_{\sN_2}(x,q)&=F_3(x,q)=S_{0,-2,-2,-1}(x,q)-S_{0,0,0,0}(x,q)+S_{2,2,1,0}(x,q),\\
f_{\sN_3}(x,q)&=F_4(x,q)=\frac{q^2}{x^2}S_{-8,-12,-8,-4}(x,q)-\frac{1}{x}S_{-2,-4,-3,-2}(x,q)
-\frac{1}{q^2}S_{0,0,0,0}(x,q)\nonumber\\
&\quad -\frac{q^2}{x^2}S_{-4,-8,-6,-3}(x,q).
\end{align}
Setting $x\mapsto 1$ and using the truth of Nandi's identities \cite{TakTsu-nandi}, we obtain:
\begin{align}
\sum_{i,j,k,\ell \ge 0}&
\frac{q^{4i^2 + 12 ij + 8ik + 4i\ell + 12j^2 + 16jk + 8j\ell + 6k^2+6k\ell+2\ell^2}}
{\left(q^2;q^2\right)_i\left(q^2;q^2\right)_j\left(q;q\right)_k\left(q;q\right)_\ell}
=
\frac{1}{\theta(q^2,q^3,q^4;q^{14})},
\\
\sum_{i,j,k,\ell \ge 0}&
\frac{q^{4i^2 + 12 ij + 8ik + 4i\ell + 12j^2 + 16jk + 8j\ell + 6k^2+6k\ell+2\ell^2}}
{\left(q^2;q^2\right)_i\left(q^2;q^2\right)_j\left(q;q\right)_k\left(q;q\right)_\ell}
\left( q^{2i+2j+k} + q^{-2j-2k-\ell} - 1 \right)
=\frac{1}{\theta(q,q^4,q^6;q^{14})},
\\
\sum_{i,j,k,\ell \ge 0}&
\frac{q^{4i^2 + 12 ij + 8ik + 4i\ell + 12j^2 + 16jk + 8j\ell + 6k^2+6k\ell+2\ell^2}}
{\left(q^2;q^2\right)_i\left(q^2;q^2\right)_j\left(q;q\right)_k\left(q;q\right)_\ell}\nonumber\\
& \times\left(
q^{-8i-12j-8k-4\ell+2}
-q^{-2i-4j-3k-2\ell}
-q^{-4i-8j-6k-3\ell+2}
-q^{-2}
\right)=
\frac{1}{\theta(q^2,q^5,q^6;q^{14})}.
\end{align}
\end{thm}

\subsection{Concluding remarks.} 
Following Takigiku and Tsuchioka~\cite{TakTsu-nandi}, it can be easily confirmed that $F_1(x,q)$ and $F_5(x,q)$ are generating functions of the following subsets of $\sN$:
\begin{align}
\sN_{F_1}&=\{ \lambda\in \sN\,\vert\, m_1(\lambda)=0, m_2(\lambda)\leq 1, m_3(\lambda)\leq 1\}
=\{ \lambda\in \sN\,\vert\, \lambda+0\in \sN\},
\\
\sN_{F_5}&=\{ \lambda\in \sN\,\vert\, m_1(\lambda)\leq 1\}
=\{ \lambda\in \sN\,\vert\, \lambda+(-2)\in \sN\}.
\end{align}
Theorem \ref{thm:nandimain} implies that $F_1(x,q)$, $F_5(x,q)$, and $F_7(x,q)$ can be written as single manifestly positive quadruple sums. It will be interesting to find a combinatorial reason behind this phenomenon: to know what makes $\sN_{F_1}, \sN_{F_5}$, and $\sN_{1}$ special. As a further avenue of research, writing $S_{0,0,0,0}(x,q) = \sum_{i,j,k,\ell\ge 0} B_{i,j,k,\ell} (x,q)$, we suggest studying  
\begin{equation}
B_{i,j,k,\ell} (x,q) = \frac{x^{2i+3j+2k+\ell} q^{4i^2 + 12 ij + 8ik + 4i\ell + 12j^2 + 16jk + 8j\ell + 6k^2+6k\ell+2\ell^2}}
{\left(q^2;q^2\right)_i\left(q^2;q^2\right)_j\left(q;q\right)_k\left(q;q\right)_\ell}
\end{equation}
to deduce which partitions belonging to $\sN_1$ are counted by each term in the quadruple sum.

\section{New mod 10 identities}
\label{sec:mod10}

Varying the linear terms in the exponent of $q$ from~\eqref{eqn:mod10_intro}, we have the following list of identities where the products are related to the
principal characters of level $4$ standard (i.e., highest-weight, integrable) $\mathrm{D}_4^{(3)}$ modules.
For notation regarding $\mathrm{D}_4^{(3)}$, see Carter's book \cite[P.\ 608]{Car-book}. 
For more information on principal characters and computational techniques for them, see the work of Bos~\cite{Bos}.
\begin{align}
\sum_{i,j,k,\ell \ge 0}
\frac{q^{2i^2 + 6 ij + 4ik + 2i\ell + 6j^2 + 8 jk + 4 j\ell + 3k^2+3k\ell+\ell^2}}
{\left(q^2;q^2\right)_i\left(q^2;q^2\right)_j\left(q;q\right)_k\left(q;q\right)_\ell} 
&=
\dfrac{1}{\theta(q,q^2,q^4;q^{10})} 
=\chi(\Omega(2\Lambda_0+\Lambda_1)),
\label{eq:mod10A}
\\
\sum_{i,j,k,\ell \ge 0}
\frac{q^{2i^2 + 6 ij + 4ik + 2i\ell + 6j^2 + 8 jk + 4 j\ell + 3k^2+3k\ell+\ell^2+2j+k+\ell}}
{\left(q^2;q^2\right)_i\left(q^2;q^2\right)_j\left(q;q\right)_k\left(q;q\right)_\ell} 
&=
\dfrac{1}{\theta(q^2,q^2,q^3;q^{10})}  
=\chi(\Omega(2\Lambda_1)),
\\
\sum_{i,j,k,\ell \ge 0}
\frac{q^{2i^2 + 6 ij + 4ik + 2i\ell + 6j^2 + 8 jk + 4 j\ell + 3k^2+3k\ell+\ell^2+2i+2j+2k}}
{\left(q^2;q^2\right)_i\left(q^2;q^2\right)_j\left(q;q\right)_k\left(q;q\right)_\ell} 
&= 
\dfrac{1}{\theta(q,q^4,q^4;q^{10})}  
=\chi(\Omega(\Lambda_0+\Lambda_2)),
\\
\sum_{i,j,k,\ell \ge 0}
\frac{q^{2i^2 + 6 ij + 4ik + 2i\ell + 6j^2 + 8 jk + 4 j\ell + 3k^2+3k\ell+\ell^2+2i+4j+3k+\ell}}
{\left(q^2;q^2\right)_i\left(q^2;q^2\right)_j\left(q;q\right)_k\left(q;q\right)_\ell} 
&=
\dfrac{1}{\theta(q^2,q^3,q^4;q^{10})}   
= \chi(\Omega(4\Lambda_0)). \label{eq:mod10D}
\end{align}
Recall that  identities for principal characters of level $3$ standard modules
of $\mathrm{D}_4^{(3)}$ were previously conjectured by the second and third authors~\cite{KanRus-idf}, and
analytic forms of these conjectures were found by Kur\c{s}ung\"{o}z~\cite{Kur-KR}. See also the thesis of the third author~\cite{Rus-thesis} for further
related identities.

We will devote the remainder of this section to a proof of these conjectures.

As in the previous section, we begin by introducing refinements of the quadruple sums, along with arbitrary linear terms in the exponent of $q$. However, this time, we use two variables, $x$ and $y$.
\begin{align}
\label{eqn:defR}
R_{A,B,C,D}(x,y,q) = \sum_{i,j,k,\ell \ge 0}
\frac{x ^{2j + k + \ell } y ^{ i+j+k}q^{2i^2 + 6 ij + 4ik + 2i\ell + 6j^2 + 8 jk + 4 j\ell + 3k^2+3k\ell+\ell^2+Ai+Bj+Ck+D\ell}}
{\left(q^2;q^2\right)_i\left(q^2;q^2\right)_j\left(q;q\right)_k\left(q;q\right)_\ell}.
\end{align}
As usual, we will drop the arguments $x,y,q$ when they are clear.

We have the following shifts of $R_{A,B,C,D}$:
\begin{align}
R_{A,B,C,D}(xq,y,q)&=R_{A,B+2,C+1,D+1}(x,y,q) \label{eqn:shiftRx},\\
R_{A,B,C,D}(x,yq,q)&=R_{A+1,B+1,C+1,D}(x,y,q) \label{eqn:shiftRy}.
\end{align}

The fundamental relations governing these sums are deduced easily.
\begin{prop}
\label{prop:mod10rels}
\begin{alignat}{3}
m_1(A,B,C,D):&\quad R_{A,B,C,D}-R_{A+2,B,C,D}-yq^{2+A}R_{A+4,B+6,C+4,D+2}&&=0 \label{eqn:relm1},\\
m_2(A,B,C,D):&\quad R_{A,B,C,D}-R_{A,B+2,C,D}-x^2yq^{6+B}R_{A+6,B+12,C+8,D+4}&&=0 \label{eqn:relm2},\\
m_3(A,B,C,D):&\quad R_{A,B,C,D}-R_{A,B,C+1,D}-xyq^{3+C}R_{A+4,B+8,C+6,D+3}&&=0 \label{eqn:relm3},\\
m_4(A,B,C,D):&\quad R_{A,B,C,D}-R_{A,B,C,D+1}-xq^{1+D}R_{A+2,B+4,C+3,D+2}&&=0 \label{eqn:relm4}.
\end{alignat}
\end{prop}

\begin{thm}
The series $R_{0,0,0,0}$ is the unique solution in $\ZZ[[x,y,q]]$ satisfying
the following:
\begin{align}
F(x,y,q) &= F(xq,y,q) + xqF\left(xq^2,y,q\right) \label{eqn:m10d1},\\
F(x,y,0) &= 1,\quad 
F(x,0,q) = \sum_{\ell\geq 0}\dfrac{q^{\ell^2}x^\ell}{(q)_\ell},\quad 
F(0,y,q) = \sum_{i\geq 0}\dfrac{q^{2i^2}y^i}{(q^2;q^2)_i}. \label{eqn:m10init}
\end{align}
\end{thm}
\begin{proof}
$R_{0,0,0,0}$ is seen to satisfy \eqref{eqn:m10init} easily.
Due to the shifts \eqref{eqn:shiftRx} and \eqref{eqn:shiftRy}, showing that $R_{0,0,0,0}$ satisfies
\eqref{eqn:m10d1} is equivalent to showing:
\begin{align}
R_{0,0,0,0} - R_{0,2,1,1} - xqR_{0,4,2,2} = 0.
\end{align}
This relation can be obtained as a linear combination of the fundamental relations \eqref{eqn:relm1}--\eqref{eqn:relm4}:
\begin{align}
-m_1&(-2,0,0,0)+m_1(-2,0,0,1)-xq\cdot m_1(0,4,2,2)+xq\cdot m_1(0,4,3,2)+m_2(0,0,0,1)
\nonumber \\
&+m_3(0,2,0,1)-xq\cdot m_3(2,4,2,2)+m_4(-2,0,0,0)-y\cdot m_4(2,6,4,2).
\end{align}

Now we prove the uniqueness. 
Suppose that $F(x,y,q)=\sum_{i,j,k\geq 0}f_{i,j,k}x^iy^jq^k$ is a solution to this system.
Note that for all $i,j,k\geq 0$, $f_{0,j,k}$, $f_{i,0,k}$ and $f_{i,j,0}$ 
are uniquely determined due to the initial conditions 
\eqref{eqn:m10init}.
For convenience, we define $f_{i,j,k}=0$ whenever any one or more of $i,j,k$ are negative.
\eqref{eqn:m10d1} translates to:
\begin{align}
f_{i,j,k}&=f_{i,j,k-i}+f_{i-1,j,k-2i+1}.
\label{eqn:frel}
\end{align}
Now we induct on $N=i+k$ and show that all $f_{i,j,k}$ are uniquely determined.
The case $N=0$ follows easily since $f_{0,j,0}$ are uniquely known due to \eqref{eqn:m10init}.
Suppose that for all triples $(i,j,k)$ with $i+k<N$, the values $f_{i,j,k}$ are determined.
Pick a triple $(I,J,K)$ such that $I+K=N$. 
If $I=0$, we know $f_{0,J,N}$ by \eqref{eqn:m10init}.
So, suppose that $I\geq 1$. 
Then, the RHS of \eqref{eqn:frel} involves two terms:
\begin{enumerate}
\item $f_{I,J,K-I}$ for which $I+(K-I) = K < N$, (since $K+I=N$ and $I\geq 1$)
\item $f_{I-1,J,K-2I+1}$ for which $(I-1)+(K-2I+1) = K - I < N$ (since $K+I=N$ and $I\geq 1$).
\end{enumerate}
This means that the RHS of \eqref{eqn:frel} has been determined already, and this determines
the LHS uniquely.
\end{proof}

\begin{thm}
Identities \eqref{eq:mod10A}--\eqref{eq:mod10D} are true.
\end{thm}
\begin{proof}
It is easy to see using chapter 7 of Andrews' text~\cite{And-book} that 
the series  
\begin{align}
\left(\sum_{i\geq 0} \dfrac{x^iq^{i^2}}{(q)_i}\right)
\left(\sum_{j\geq 0} \dfrac{y^jq^{2j^2}}{(q^2;q^2)_j}\right)
\end{align}
satisfies \eqref{eqn:m10d1} and \eqref{eqn:m10init}.
Thus, due to uniqueness, we have:
\begin{align}
R_{0,0,0,0}(x,y,q)=\left(\sum_{i\geq 0} \dfrac{x^iq^{i^2}}{(q)_i}\right)
\left(\sum_{j\geq 0} \dfrac{y^jq^{2j^2}}{(q^2;q^2)_j}\right).
\end{align}
By the Rogers--Ramanujan identities \cite[Ch.\ 7]{And-book}, we now have:
\begin{align}
R_{0,0,0,0}(1,1,q)=\dfrac{1}{\theta(q;q^5)\theta(q^2;q^{10})},\\
R_{0,0,0,0}(q,1,q)=\dfrac{1}{\theta(q^2;q^5)\theta(q^2;q^{10})},\\
R_{0,0,0,0}(1,q^2,q)=\dfrac{1}{\theta(q;q^5)\theta(q^4;q^{10})},\\
R_{0,0,0,0}(q,q^2,q)=\dfrac{1}{\theta(q^2;q^5)\theta(q^4;q^{10})}.
\end{align}
This immediately proves~\eqref{eq:mod10A}--\eqref{eq:mod10D}.
\end{proof}

\appendix

\section{Modified Murray--Miller algorithm}
\label{sec:mmrec}
We follow the pseudo-code given by Takigiku and Tsuchioka~\cite{TakTsu-nandi} for the modified Murray--Miller algorithm.
Our Maple program \texttt{murraymiller.mw} is used as follows.

We begin by importing the Linear Algebra package, and then loading our Maple program.
\begin{quote}
\begin{verbatim}
with(LinearAlgebra);
read(`murraymiller.txt`);
\end{verbatim}
\end{quote}
This file provides the recurrence matrix of \eqref{eqn:Ndiffeq}:
\begin{quote}
\begin{verbatim}
tt;
\end{verbatim}
\end{quote}
which equals:
\begin{align*}
\begin{bmatrix}
1&x{q}^{2}&{x}^{2}{q}^{4}&xq&{x}^{2}{q}^{2}&0&0\\
0&x{q}^{2}&0&0&0&1&0\\
0&0&0&0&0&0&1\\
0&x{q}^{2}&0&xq&0&1&0\\ 
0&0&0&0&x{q}^{2}&0&1\\ 
1&x{q}^{2}&{x}^{2}{q}^{4}&xq&0&0&0\\
1&x{q}^{2}&{x}^{2}{q}^{4}&0&0&0&0
\end{bmatrix} 
\end{align*}
Recall that the indices of this matrix are $0,1,\dots, 5, 7$.
The latter procedures assume that we are trying to find 
a higher order difference equation satisfied by the function corresponding
to the first index.
So, to find a recurrence satisfied by $F_7$, we first exchange the indices
$0$ and $7$ by exchanging the rows and columns.
\begin{quote}
\begin{verbatim}
ttF7:=ColumnOperation(RowOperation(tt,[1,7]),[1,7]);
\end{verbatim}
\end{quote}
This gives us:
\begin{align*}
\mathtt{ttF7}:=\begin{bmatrix} 0&x{q}^{2}&{x}^{2}{q}^{4}&0&0&0&1\\ 
0&x{q}^{2}&0&0&0&1&0\\
1&0&0&0&0&0&0\\ 
0&x{q}^{2}&0&xq&0&1&0\\ 
1&0&0&0&x{q}^{2}&0&0\\ 
0&x{q}^{2}&{x}^{2}{q}^{4}&xq&0&0&1\\ 
0&x{q}^{2}&{x}^{2}{q}^{4}&xq&{x}^{2}{q}^{2}&0&1
\end{bmatrix}
\end{align*}
Now we put this matrix in ``standard form'':
\begin{quote}
\begin{verbatim}
mmttF7:=murraymiller(ttF7,2);
\end{verbatim}
\end{quote}
Here the second argument \texttt{2} corresponds to the shift $x\mapsto xq^2$.
The output is:
\begin{align*}
\mathtt{mmttF7}:=
[5, 
{\renewcommand{\arraystretch}{1.2}
\begin{bmatrix} 
0&1&0&0&0&0&0\\ 
{x}^{2}&x+1&1&0&0&0&0\\ 
-{\frac {{x}^{2} \left( x-1 \right) }{{q}^{2}}}&{\frac {x}{q}}&{q}^{-1}&1&0&0&0\\ 
-{\frac {{x}^{2}}{{q}^{3}}}&-{\frac {x \left( {q}^{2}-x \right) }{{q}^{4}}}&-{\frac {{q}^{2}-x}{{q}^{4}}}&-{\frac {{q}^{2}-x}{{q}^{3}}}&1&0&0\\
-{\frac {{x}^{3}}{{q}^{7}}}&0&0&0&{\frac {x}{{q}^{4}}}&0&0\\ 
0&1&0&{\frac {q}{x-1}}&-{\frac {{q}^{3}}{x \left( x-1 \right) }}&0&0\\
0&1&0&{\frac {q}{x-1}}&-{\frac {{q}^{3}}{x-1}}&0&0
\end{bmatrix}
]
}
\end{align*}
Here, the second entry in the output is the recurrence matrix put into a standard form, 
and the first entry \texttt{5} denotes that the first $5\times 5$ block is 
to be used to find the recurrence.
We thus use:
\begin{quote}
\begin{verbatim}
mmrec(mmttF7[2][1..5,1..5],2,g);
\end{verbatim}
\end{quote}
Here, again the first argument is the relevant portion of the matrix in the standard form, the second argument is the shift $x\mapsto xq^2$,
and the third argument is the dummy variable to be used in the recurrence.
The output is:
\begin{quote}
\begin{align*}
&\mathtt{
{\frac {{x}^{3} \left( x{q}^{2}-{q}^{2}-{x}^{2}+x \right) g \left( x{q
}^{2} \right) }{{q}^{9}}}
-{\frac {{x}^{3} \left( {q}^{4}+{q}^{3}-x{q}^
{2}+{q}^{2}-xq-x \right) g \left( x \right) }{{q}^{13}}}}\\
&\mathtt{+{\frac {{x}^{2} \left( {q}^{5}x+{q}^{4}-x{q}^{2}-xq-x \right) }{{q}^{14}}g \left( {\frac {x}{{q}^{2}}} \right) }
-{\frac {x \left( {q}^{8}+{q}^{5}x+x{q}^{4}+x{q}^{3}-x \right) }{{q}^{12}}g \left( {\frac {x}{{q}^{4}}}\right) }}\\
&\mathtt{+{\frac {{q}^{6}+x{q}^{2}+xq+x}{{q}^{6}}g \left( {\frac {x}{
{q}^{6}}} \right) }-g \left( {\frac {x}{{q}^{8}}} \right) 
}
\end{align*}
\end{quote}
This is equivalent to \eqref{eqn:recF7} upon shifting $x\mapsto xq^8$.

We repeat for $F_1$:
\begin{quote}
\begin{verbatim}
ttF1:=ColumnOperation(RowOperation(tt,[1,2]),[1,2]);
mmttF1:=murraymiller(ttF1,2);
mmrec(mmttF1[2][1..6,1..6],2,g);
\end{verbatim}
\end{quote}
The output is:
\begin{quote}
\begin{align*}
&\mathtt{ 
{\frac {{x}^{3} \left( x{q}^{2}-{q}^{2}-{x}^{2}+x \right) g \left( x
 \right) }{{q}^{13}}}
 -{\frac {{x}^{3} \left( {q}^{5}-x{q}^{3}+{q}^{3}+
{q}^{2}-xq-x \right) }{{q}^{16}}g \left( {\frac {x}{{q}^{2}}} \right) 
}}\\
&\mathtt{+{\frac {{x}^{2} \left( {q}^{7}-{q}^{5}+{q}^{4}-x{q}^{2}-x \right) }{
{q}^{16}}g \left( {\frac {x}{{q}^{4}}} \right) }
-{\frac {x \left( {q}^
{6}+{q}^{5}-{q}^{4}+x{q}^{2}+x \right) }{{q}^{11}}g \left( {\frac {x}{
{q}^{6}}} \right) }}\\
&
\mathtt{+{\frac {{q}^{8}+x{q}^{3}+x{q}^{2}+x}{{q}^{8}}g
 \left( {\frac {x}{{q}^{8}}} \right) }-g \left( {\frac {x}{{q}^{10}}}
 \right)} 
\end{align*}
\end{quote}
which is equivalent to \eqref{eqn:recF1} under $x\mapsto xq^{10}$.

For $F_5$:
\begin{quote}
\begin{verbatim}
ttF5:=ColumnOperation(RowOperation(tt,[1,6]),[1,6]);
mmttF5:=murraymiller(ttF5,2);
mmrec(mmttF5[2][1..6,1..6],2,g);
\end{verbatim}
\end{quote}
The output is:
\begin{quote}
\begin{align*}
&\mathtt{{\frac {{x}^{3} \left( x{q}^{2}-{q}^{2}-{x}^{2}+x \right) g \left( x
 \right) }{{q}^{19}}}
 -{\frac {{x}^{3} \left( {q}^{4}+{q}^{3}-x{q}^{2}+
{q}^{2}-xq-x \right) }{{q}^{21}}g \left( {\frac {x}{{q}^{2}}} \right) 
}}\\
&\mathtt{+{\frac { \left( {q}^{9}-{q}^{7}+x{q}^{5}+{q}^{4}-x{q}^{2}-xq-x
 \right) {x}^{2}}{{q}^{20}}g \left( {\frac {x}{{q}^{4}}} \right) }
}\\
&\mathtt{ -{
\frac {x \left( {q}^{10}+{q}^{9}-{q}^{7}+x{q}^{5}+x{q}^{4}+x{q}^{3}-x
 \right) }{{q}^{16}}g \left( {\frac {x}{{q}^{6}}} \right) }
 +{\frac {{q
}^{8}+x{q}^{2}+xq+x}{{q}^{8}}g \left( {\frac {x}{{q}^{8}}} \right) }-g
 \left( {\frac {x}{{q}^{10}}} \right) 
} 
 \end{align*}
\end{quote}
which is equivalent to \eqref{eqn:recF5} under $x\mapsto xq^{10}$.
\section{Proof verification}
\label{sec:proofverification}

We explain the Maple program \texttt{checknandi.mw} that verifies
\eqref{eqn:recS_F1}--\eqref{eqn:recS_F7} as explicit linear combinations of 
fundamental relations \eqref{eqn:reln1}--\eqref{eqn:reln4}.

We begin by defining \eqref{eqn:reln1}--\eqref{eqn:reln4}:
\begin{quote}
\begin{verbatim}
N1 := (A,B,C,D) -> S(A,B,C,D)
                   -x^2*q^(4+A)*(1+q^2)*S(A+8,B+12,C+8,D+4)
                   -S(A+4,B,C,D)
                   +x^4*q^(18+2*A)*S(A+16,B+24,C+16,D+8):

N2 := (A,B,C,D) -> S(A,B,C,D)
                   -S(A,B+2,C,D)
                   -x^3*q^(12+B)*S(A+12,B+24,C+16,D+8): 

N3 := (A,B,C,D) -> S(A,B,C,D)
                   -x^2*q^(6+C)*S(A+8,B+16,C+12,D+6)-S(A,B,C+2,D)
                   -x^2*q^(7+C)*S(A+8,B+16,C+12,D+6)
                   +x^4*q^(25+2*C)*S(A+16,B+32,C+24,D+12):

N4 := (A,B,C,D) -> S(A,B,C,D)
                   -x*q^(2+D)*S(A+4,B+8,C+6,D+4)-S(A,B,C,D+2)
                   -x*q^(3+D)*S(A+4,B+8,C+6,D+4)
                   +x^2*q^(9+2*D)*S(A+8,B+16,C+12,D+8):
\end{verbatim}
\end{quote}
We now read the file that contains a linear combination of 
\eqref{eqn:reln1}--\eqref{eqn:reln4} which when expanded is supposed
to yield \eqref{eqn:recS_F1}.
\begin{quote}
\begin{verbatim}
F1rel := parse(FileTools[Text][ReadFile]("F1.txt")):
\end{verbatim}
\end{quote}
Simply simplifying this entire expression takes too long.
Thus, we collect all like $S$ terms together and simplify the coefficients.
\begin{quote}
\begin{verbatim}
collect(F1rel, S, simplify);
\end{verbatim}
\end{quote}
Naturally, most coefficients are $0$ and get dropped from the expression.
The output is:
\begin{quote}
\begin{verbatim}
-x^3*q^16*(q^11*x+q^9*x+q^8*x-q^3-q-1)*S(18, 26, 17, 8)
+x*q^3*(q^8*x+q^6*x+q^2+q-1)*S(10, 14, 9, 4)
+x^3*q^19*(q^18*x^2-q^10*x-q^8*x+1)*S(22, 32, 21, 10)
+S(2, 2, 1, 0)+(-q^5*x-q^4*x-q^2*x-1)*S(6, 8, 5, 2)
+x^2*q^8*(q^8*x+q^6*x-q^3+q-1)*S(14, 20, 13, 6)
\end{verbatim}
\end{quote}
which is exactly \eqref{eqn:recS_F1}.
We repeat the process for \eqref{eqn:recS_F5}:
\begin{quote}
\begin{verbatim}
F5rel := parse(FileTools[Text][ReadFile]("F5.txt")):
collect(F5rel, S, simplify);
\end{verbatim}
\end{quote}
and the answer matches \eqref{eqn:recS_F5}:
\begin{quote}
\begin{verbatim}
S(0, -2, -2, -1)
-x^2*q^4*(q^11*x-q^8*x-q^7*x-q^6*x+q^5-q^3+1)*S(12, 16, 10, 5)
-(q^10*x+q^9*x+q^8*x-q^2-q-1)*q^11*x^3*S(16, 22, 14, 7)
+(q^8*x+q^7*x+q^6*x-q^3*x+q^3+q^2-1)*q*x*S(8, 10, 6, 3)
+(q^18*x^2-q^10*x-q^8*x+1)*q^13*x^3*S(20, 28, 18, 9)
+(-q^4*x-q^3*x-q^2*x-1)*S(4, 4, 2, 1)
\end{verbatim}
\end{quote}
For \eqref{eqn:recS_F7}:
\begin{quote}
\begin{verbatim}
F7rel := parse(FileTools[Text][ReadFile]("F7.txt")):
collect(F7rel, S, simplify);
\end{verbatim}
\end{quote}
and the answer matches \eqref{eqn:recS_F7}:
\begin{quote}
\begin{verbatim}
S(0, 0, 0, 0)
+(q^5*x+q^4*x+q^3*x-x+1)*q^4*x*S(8, 12, 8, 4)
+x^3*q^17*(q^14*x^2-q^8*x-q^6*x+1)*S(20, 30, 20, 10)
-x^2*q^6*(q^9*x-q^6*x-q^5*x-q^4*x+1)*S(12, 18, 12, 6)
+(-q^4*x-q^3*x-q^2*x-1)*S(4, 6, 4, 2)
-x^3*q^13*(q^8*x+q^7*x+q^6*x-q^2-q-1)*S(16, 24, 16, 8)
\end{verbatim}
\end{quote}

\providecommand{\oldpreprint}[2]{\textsf{arXiv:\mbox{#2}/#1}}\providecommand{\preprint}[2]{\textsf{arXiv:#1
  [\mbox{#2}]}}

\end{document}